\documentclass{birkjour}
 \newtheorem{thm}{Theorem}[section]
 \newtheorem{cor}[thm]{Corollary}
 \newtheorem{lem}[thm]{Lemma}
 \newtheorem{prop}[thm]{Proposition}
 \theoremstyle{definition}
 \newtheorem{defn}[thm]{Definition}
 \theoremstyle{remark}
 \newtheorem{rem}[thm]{Remark}
 \newtheorem{ex}[thm]{Example}
 \theoremstyle{remark}
 \numberwithin{equation}{section}

\usepackage{hyperref}
\usepackage{enumitem}

\newcommand{\normmm}[1]{{\left\vert\kern-0.25ex\left\vert\kern-0.25ex\left\vert #1
		\right\vert\kern-0.25ex\right\vert\kern-0.25ex\right\vert}}
\newcommand{\normm}[1]{{\vert\kern-0.25ex\vert\kern-0.25ex\vert #1
		\vert\kern-0.25ex\vert\kern-0.25ex\vert}}
	
\begin{document}

%-------------------------------------------------------------------------
% editorial commands: to be inserted by the editorial office
%
%\firstpage{1} \volume{228} \Copyrightyear{2004} \DOI{003-0001}
%\seriesextraline{This is the Concrete Title of this Book\br H.E. R and S.T.C. W, Eds.}
%\firstpage{1}

\title[Relations among the notions of various kinds of stability and applications]
 {The relations among the notions of various kinds of stability and their applications}

%----------Author 1
\author[Tiexin Guo]{Tiexin Guo*}

\address{%
School of Mathematics and Statistics\\
Central South University\\
ChangSha 410083\\
China}

\email{tiexinguo@csu.edu.cn}

\thanks{
	*Corresponding author.\\
	This research is supported by the National Natural Science Foundation of China (Grant Nos.
	12371141,11971483) and the Provincial Natural Science Foundation of Hunan (Grant No. 2023JJ30642).}
	
%----------Author 2
\author{Xiaohuan Mu}
\address{%
	School of Mathematics and Statistics\\
	Central South University\\
	ChangSha 410083\\
	China}
	
\email{xiaohuanmu@163.com}

%----------Author 3
\author{Qiang Tu}
\address{%
	School of Mathematics and Statistics\\
	Central South University\\
	ChangSha 410083\\
	China}

\email{qiangtu126@126.com}



\date{\today}
%----------additions附言

\begin{abstract}
First, we prove that a random metric space can be isometrically embedded into a complete random normed module, as an application it is easy to see that the notion of $d$-$\sigma$-stability in a random metric space can be regarded as a special case of the notion of $\sigma$-stability in a random normed module; as another application we give the final version of the characterization for a $d$-$\sigma$-stable random metric space to be stably compact. Second, we prove that an $L^{p}$-normed $L^{\infty}$-module is exactly generated by a complete random normed module so that the gluing property of an $L^{p}$-normed $L^{\infty}$-module can be derived from the $\sigma$-stability of the generating random normed module, as applications almost all the basic theory of module duals can be obtained from the theory of random conjugate spaces. Third, we prove that a random normed space is order complete iff it is $(\varepsilon,\lambda)$-complete, as an application it is proved that the $d$-decomposability of an order complete random normed space is exactly its $d$-$\sigma$-stability. Finally, we prove that an equivalence relation on the product space of a nonempty set $X$ and a complete Boolean algebra $B$ is regular iff it can be induced by a $B$-valued Boolean metric on $X$, as an application it is proved that a nonempty subset of a Boolean set $(X,d)$ is universally complete iff it is a $B$-stable set defined by a regular equivalence relation.

\end{abstract}

%%% ----------------------------------------------------------------------
\maketitle
%%% ----------------------------------------------------------------------
%\tableofcontents

%###############################################################################################
%###############################################################################################
\section{Introduction}
%###############################################################################################
%###############################################################################################

The notions of various kinds of stability were introduced and have played crucial roles in the study of scattered topics in analysis and geometry. The purpose of this paper is to unify these notions by analysing the relations among them and give some applications. For the clarity and convenience of readers, we will divide this introduction into the following four subsections to recapitulate the backgrounds for these notions and their roles in various kinds of topics in analysis and geometry, while the main results of this paper are briefly summarized but the concrete notions of stability will be given when they are needed in the subsequent text of this paper.

%##################################################################################################
\subsection{On $\sigma$-stability and $d$-$\sigma$-stability in random functional analysis}
%##################################################################################################

 Random functional analysis is based on the idea of randomizing the traditional space theory of functional analysis. Such an idea dates back to the theory of probabilistic metric spaces initiated by Menger, Schweizer and Sklar and the others \cite{SS}. Strictly speaking, the work on random functional analysis began with the study of random metric spaces and random normed spaces whose original definitions were introduced in \cite[Chapters 9 and 15]{SS}. The development of random normed spaces had been in a stagnant state for a long time mainly because random normed spaces, which are often endowed with the $(\varepsilon,\lambda)$-topology, are not locally convex and hence the theory of conventional conjugate spaces fails to work for them. Our breakthrough first came in \cite{Guo6a} where Guo first gave an equivalent formulation of the notion of a random normed space (see also Definition \ref{def.2.1} and Definition \ref{def.2.4} in this paper), introduced an almost everywhere (briefly, a.e.) bounded random linear functional and established the corresponding Hahn-Banach theorem (see Example \ref{exm.2.12}), which leads to the development of random conjugate spaces. Subsequently, the notions of random normed and inner product modules were further introduced in \cite {Guo6aa,Guo1} so that the theory of random conjugate spaces obtained a fast development, for example, the representation theory of random conjugate spaces was given in \cite{Guo7a,GY,Guo5}, the characterizations of random reflexivity under the framework of random conjugate spaces were deeply studied in \cite{GL,Guo4}, the Helly theorem in random normed modules was also established by characterizing the dimensional structure of finitely generated $L^{0}$-modules in \cite{GS} and the geometric theory of random normed modules was deeply studied in \cite{GZ}. In 1999, the notion of a random locally convex module was introduced by
 Guo and the theory of random conjugate spaces for random locally convex modules was also widely developed, see \cite{Guo1,Guo2,Guo3,Guo6,GZWG,GZWYYZ} and the reference literature therein. To sum up, random functional analysis is functional analysis based on random metric spaces, random normed modules and random locally convex modules, which were developed under the $(\varepsilon,\lambda)$-topology before 2009. The basic concepts and results of random functional analysis well developed under the $(\varepsilon,\lambda)$-topology have played an essential role in the development of nonsmooth differential geometry on metric measure spaces, see \cite{Gigl,LP1,LP2} and also Section \ref{sec.1.2} for details.

In 2009, locally $L^0$-convex modules were introduced in \cite{FKV} in order to establish a generalized  convex analysis suitable for the study of conditional convex risk measures. With the notion of a locally $L^0$-convex module, the locally $L^0$-convex topology for a random locally convex module was also presented in \cite{FKV}, it is later showed independently in \cite{WG} and in \cite{Zap}that the truly useful part of the theory of a locally $L^0$-convex module amounts to the theory of a random locally convex module endowed with the locally $L^0$-convex topology. The $(\varepsilon,\lambda)$-topology is an abstract generalization of the usual topology of convergence in probability measure and has been widely used in functional analysis, probability theory and mathematical finance. The locally $L^0$-convex topology is stronger than the $(\varepsilon,\lambda)$-topology so that the locally $L^0$-convex topology can ensure most of the $L^0$-convex sets in question to have nonempty interiors, which makes it possible to establish the continuity and subdifferentiability theorems for $L^0$-convex functions, see \cite{FKV,GZWYYZ} for details. The locally $L^0$-convex topology is, however, too strong to ensure it to be a linear topology so that it is difficult to establish the connection between the theory of locally $L^0$-convex modules and the corresponding theory of conventional locally convex spaces. Thus, the $(\varepsilon,\lambda)$-topology and locally $L^0$-convex topology have their respective advantages and disadvantages, it is desirable to combine their advantages so that a perfect random convex analysis truly suitable for mathematical finance can be established, see e.g. \cite{GZWYYZ,GZZ1,GZZ2}, and a key step towards the goal is naturally to establish the connection between some basic results derived from the two kinds of topologies, which was first done in \cite{Guo3} where the notion of $\sigma$-stable sets (namely, $\sigma$-stability) in an $L^0$-module was presented and has played a crucial role in establishing the above-stated connection.  In fact, the paper \cite{Guo3} also advocates a subsequent developing model of random functional analysis by simultaneously considering the above-stated two kinds of topologies and in the process the assumption on $\sigma$-stability for random locally convex modules or their subsets in question are ubiquitous.

With the deep development of the theory of random conjugate spaces, Guo \cite{Guo4} earlier found that some basic theorems involving $w^*$-compactness and weak compactness for normed spaces are no longer valid for general complete random normed modules under random $w^{*}$-topology and random weak topology. On the other hand, to provide a simplifying proof of no-arbitrage criteria, Kabanov and Stricker \cite{KS} proved the randomized Bolzano-Weierstrass theorem, which states that every almost surely bounded sequence of random variables with values in the Euclidean space admits an almost surely convergent randomized subsequence although it does not admit any almost surely convergent subsequence, which also motivates the subsequent development of the theory of random sequential compactness \cite{GWCXY,JKZ}. Motivated by the work \cite{Guo3,Guo4,KS,FKV}, Drapeau, et.al \cite{DJKK} presented the notions of conditional sets and conditional topology with an attempt to provide the theory of conditional compactness suitable for the further development of random functional analysis or more general conditional analysis. Conditional sets and their operations (conditional join, meet and complement, et.al) are, however, complicated and in fact can also be expressed in terms of conventional set theory, which yields the notions of B-stable sets and stably compactness \cite{JZ,GWT}. It should be pointed out that the notion of B-stable sets (or equivalently, conditional sets) can be regarded as an important and abstract generalization of that of $\sigma$-stable sets. Likewise, every mathematical construct in conditional set theory depends on the notion of $B$-stability.

As mentioned above, the notion of $\sigma$-stability was introduced in \cite{Guo3} by means of the $L^0$-module multiplication and thus is purely algebraic, which has made the theory of random normed modules and random locally convex modules obtain a significant and profound advance, see e.g. \cite{Guo3,GZWG,GZWYYZ,GWCXY,GWT} for details. Contrasted with this, the theory of random metric spaces as a random generalization of ordinary metric spaces had not made much substantive progress for a quite long time although random metric spaces were earliest presented in random functional analysis. With the notion of $d$-$\sigma$-stability for subsets of a random metric space presented in \cite{JKZ}, such a situation was beginning to change, a series of deep developments on random metric spaces were achieved \cite{GWYZ}. Although similar to $\sigma$-stability, $d$-$\sigma$-stability depends on random
metric structure rather than an $L^0$-module structure. In this paper, we will prove that every random metric space can be isometrically embedded into a $\mathcal{T}_{\varepsilon,\lambda}$-complete random normed module so that the notion of $d$-$\sigma$-stability can be regarded as a special case of that of $\sigma$-stability. As applications of this isometric embedding, we can transfer many important results well established for random normed modules onto random metric spaces, and in particular we give the final version of the characterization for a random metric space to be stably compact, see Theorem \ref{thm.2.21} for details.

%##################################################################################################
\subsection{On the gluing property and $\sigma$-stability}\label{sec.1.2}
%##################################################################################################

The notion of $L^p$-normed $L^{\infty}$-modules was introduced by Gigli in \cite{Gigl} as one of the three pillars of nonsmooth differential geometry on metric measure spaces, where the main purpose is to provide a robust functional-analytic framework suitable for constructing effective notions of $1$-forms and vector fields in the setting of metric measure spaces, see \cite{Gigl1,Gigl,LP1,LP2,GLP} for details. In fact, the connection between $L^p$-normed $L^{\infty}$-modules and $\mathcal{T}_{\varepsilon,\lambda}$-complete random normed modules has been thoroughly studied in \cite{Gigl} in the way that given an $L^p$-normed $L^{\infty}$-module in advance, a $\mathcal{T}_{\varepsilon,\lambda}$-complete random normed module is then obtained as a completion of the $L^p$-normed $L^{\infty}$-module with respect to the topology of convergence locally in measure, it is in such a context that an equivalent notion of a random normed module was independently introduced in Gigli's spectacular work \cite{Gigl1,Gigl} in the name of an $L^{0}$-normed $L^{0}$-module. In this paper, we adopt a reverse way as in \cite{GL,Guo5} that given a $\mathcal{T}_{\varepsilon,\lambda}$-complete random normed module in advance, an $L^p$-normed $L^{\infty}$-module is then obtained as the abstract $L^p$-space generated by the random normed module, in fact, every $L^p$-normed $L^{\infty}$-module can be also obtained in such a way, see Theorem \ref{thm.3.6} for details.
Now that every $L^p$-normed $L^{\infty}$-module is exactly generated by a $\mathcal{T}_{\varepsilon,\lambda}$-complete random normed module, one can easily see that the gluing property of an $L^p$-normed $L^{\infty}$-module can be derived from the $\sigma$-stability of the generating random normed module, this paper will further show that all the results of module duals for $L^p$-normed $L^{\infty}$-modules, which were obtained in \cite{Gigl}, can be derived from the well-established theory of random conjugate spaces for random normed modules in \cite{Guo5,GL}, and in particular by making use of the Hahn-Banach theorem on random normed modules we can still establish the Hahn-Banach theorem for the elements of module duals for $L^p$-normed $L^{\infty}$-modules although there is not yet the corresponding result for module duals of general $L^{\infty}$-modules. Besides, Lu\v{c}i\'{c} and Pasqualetto \cite{LP2} recently have thoroughly studied measurable Banach bundles. Given a measurable Banach bundles $E$, let $\Gamma_{0}(E)$ be the set of equivalence classes of strongly measurable sections of $E$, then $\Gamma_{0}(E)$ is a $\mathcal{T}_{\varepsilon,\lambda}$-complete random normed module, further let $\Gamma_{p}(E)$ be the $L^p$-normed $L^{\infty}$-module generated by $\Gamma_{0}(E)$ for any $p\in [1,+\infty]$, Lu\v{c}i\'{c} and Pasqualetto gave the two elegant representation theorems of the module dual of $\Gamma_{p}(E)$, see Theorem 3.8 and Proposition 3.10 of \cite{LP2}. In this paper, by explicitly representing the random conjugate space of $\Gamma_{0}(E)$ we give a simpler proof of Theorem 3.8 and Proposition 3.10 of \cite{LP2}, in fact, their representation theorems both can be regarded as a special case of Corollary \ref{coro.3.10} of this paper.

%##################################################################################################
\subsection{On $d$-decomposability and $d$-$\sigma$-stability in order complete random normed spaces }
%##################################################################################################

The notion of a lattice-normed space was introduced by Kantorovich \cite{Kant,KVP} in 1939 (prior to the development of random functional analysis) in connection with the development of the general theory of approximation methods, where the notion of $d$-decomposable lattice norm was presented. $d$-decomposable order complete lattice-normed spaces (namely, Banach- Kantorovich spaces) were deeply studied by Kusraev \cite{Kurs}. It is obvious that random normed spaces are a special class of lattice-normed spaces. In this paper we will prove that order completeness and $\mathcal{T}_{\varepsilon,\lambda}$-completeness are equivalent for random normed spaces so that a $d$-decomposable order complete random normed space is automatically a $\mathcal{T}_{\varepsilon,\lambda}$-complete random normed module, which also implies that $d$-decomposability automatically becomes $d$-$\sigma$-stability for an order complete random normed space.

%##################################################################################################
\subsection{On universal completeness and $B$-stability}
%##################################################################################################

Last but not the least, we study the relation between the universal completeness defined by a Boolean valued metric and the $B$-stability defined by a regular equivalence relation. The notion of a Boolean valued metric (or, a $B$-set) was introduced in 1950s, see \cite{KK} and the references therein, the universal completeness defined by a Boolean metric has played an ubiquitous role in carrying out the mathematical constructs in Boolean valued analysis. The $B$-stability of a nonempty set $X$ with respect to a Boolean algebra $B$ is defined by a regular equivalence relation on $X\times B$. In this paper, it is proved that an equivalence relation on $X\times B$ is regular iff it can be induced by a $B$-valued Boolean metric on $X$, and thus the notions of universal completeness and $B$-stability are equivalent to each other, which also means that all the other notions of stability involved in this paper can be unified to a single one -- universal completeness or $B$-stability.

The remainder of this paper is organized as follows. Section 2 is devoted to proving that $d$-$\sigma$-stability can be regarded a special case of
$\sigma$-stability by isometrically embedding a random metric space into a $\mathcal{T}_{\varepsilon,\lambda}$-complete random normed module, where the final version of the  characterization for a $d$-$\sigma$-stable random metric space to be stably compact is given as an application. Section 3 is devoted to proving that the gluing property can be derived from $\sigma$-stability by checking that every $L^p$-normed $L^\infty$-module can be exactly generated by a $\mathcal{T}_{\varepsilon,\lambda}$-complete random normed module, where the Hahn-Banach theorem and its geometric form in the sense of module duals of $L^{p}$-normed $L^{\infty}$-modules are given as applications of the theory of random conjugate spaces of $RN$ spaces. Besides, some other applications of the theory of random conjugate spaces to module duals are also given in this section. Section 4 is devoted to proving that $d$-decomposability automatically becomes $d$-$\sigma$-stability in order complete random normed spaces by checking that order completeness is equivalent to $\mathcal{T}_{\varepsilon,\lambda}$-completeness. Finally, Section 5 is devoted to proving that universal completeness is equivalent to $B$-stability by checking that a Boolean valued metric amounts to a regular equivalence relation.

Throughout this paper, $\mathbb{N}$ always stands for the set of positive integers, $\mathbb{K}$ for the scalar field $\mathbb{R}$ of real numbers or $\mathbb{C}$ of complex numbers, $(\Omega,\mathcal{F},\mu)$ a nontrivial $\sigma$-finite measure space (where ``nontrivial'' means $\mu(\Omega)>0$), and $L^{0}(\mathcal{F},\mathbb{K})$ the usual algebra over $\mathbb{K}$ of equivalence classes of measurable functions from $(\Omega,\mathcal{F},\mu)$ to
$\mathbb{K}$.

Specially, we simply denote $L^{0}(\mathcal{F},\mathbb{R})$ by $L^{0}(\mathcal{F})$. Besides, $\bar{L}^{0}(\mathcal{F})$ always denotes the set of equivalence classes of extended real-valued measurable functions defined on $(\Omega,\mathcal{F},\mu)$. Proposition \ref{pro.1.1} is known from \cite{DS}, which is frequently used in this paper.

\begin{prop}\label{pro.1.1}
	$\bar{L}^0(\mathcal{F})$ is a complete lattice under the partial order $\leq$: $\xi \leq \eta$ iff $\xi^0(\omega) \leq \eta^0(\omega)$ for almost all $\omega$ in $\Omega$ (briefly, a.e.), where $\xi^0 $ and $\eta^0$ are arbitrarily chosen representatives, respectively. We always denote by $\bigvee A$ and $\bigwedge A$ the supremum and infimum of a nonempty subset $A$ in $\bar{L}^0(\mathcal{F})$, respectively. Furthermore, $(\bar{L}^0(\mathcal{F}), \leq)$ possesses the  following nice properties:
	\begin{itemize}
		\item [(1)] There are two sequences $\{a_n, n \in \mathbb{N} \}$ and $\{b_n, n \in \mathbb{N} \}$ in $A$ such that $\bigvee_{n \geq 1} a_n = \bigvee A$ and $\bigwedge_{n \geq 1}b_n = \bigwedge A$.
		\item [(2)] If $A$ is directed upwards (or downwards), then $\{a_n, n \in \mathbb{N} \}$ (correspondingly, $\{ b_n, n \in \mathbb{N} \}$) can be chosen as nondecreasing (nonincreasing).
		\item [(3)]	$(L^0(\mathcal{F}), \leq)$, as a sublattice of $\bar{L}^0(\mathcal{F})$, is naturally Dedekind complete.	
	\end{itemize}
\end{prop}

In the sequel of this paper, we also employ the following notations and terminologies:

$\bar{L}_{+}^{0}(\mathcal{F})=\{\xi\in \bar{L}^{0}(\mathcal{F}):\xi\geq 0\}$.

$L_{+}^{0}(\mathcal{F})=\{\xi\in L^{0}(\mathcal{F}):\xi\geq 0\}$.

As usual, for any $\xi$ and $\eta$ in $\bar{L}^0(\mathcal{F})$, $\xi > \eta$  means $\xi \geq \eta$ but $\xi \neq \eta$, where as $\xi > \eta$ on $A$ for any $A\in \mathcal{F}$ means $\xi^0(\omega)>\eta^0(\omega)$ a.e. on $A$, where $\xi^0$ and $\eta^0$ are arbitrarily chosen representatives of $\xi$ and $\eta$, respectively.

$\bar{L}_{++}^{0}(\mathcal{F})=\{\xi\in \bar{L}^{0}(\mathcal{F}):\xi>0~\text{on}~\Omega\}$.

$L_{++}^{0}(\mathcal{F})=\{\xi\in L^{0}(\mathcal{F}):\xi>0~\text{on}~\Omega\}$.

Let $B$ be a complete Boolean algebra, a nonempty subset $\{a_{i},i\in I\}$ of $B$ is called a partition of unity if $\vee_{i\in I}a_{i}=1$ and $a_{i}\wedge a_{j}= 0$ where $i\neq j$.

For the $\sigma$-finite measure space $(\Omega,\mathcal{F},\mu)$, two elements $A$ and $D$ is said to be equivalent if $\mu(A\bigtriangleup D)=0$ (where $A\bigtriangleup D=(A\verb|\|D)\cup (D\verb|\| A)$ stands for the symmetric difference of $A$ and $D$). Throughout this paper, for any $A\in \mathcal{F}$, we always use the corresponding lowercase letter $a$ for the equivalence class $[A]$ of $A$, and we also use $B_{\mathcal{F}}$ for the complete Boolean algebra of equivalence classes of elements of $\mathcal{F}$, well known as the measure algebra associated with $(\Omega,\mathcal{F},\mu)$. Let us further recall that, for any $a=[A]$ and $d=[D]$ in $B_{\mathcal{F}}$, $a\leq d$ iff $\mu(A\verb|\|D)=0$; moreover, $0=[\emptyset],~1=[\Omega]$, and the complement $a^{c}$ of $a=[A]$ is $[A^{c}]$. It should be also noticed that $\{i\in I:a_{i}>0\}$ must be at most countable for any partition $\{a_{i},i\in I\}$ of unity in $B_{\mathcal{F}}$, and thus when we speak of partitions of unity in $B_{\mathcal{F}}$ we only need to consider the partitions of unity which are at most countable.

Here, we would like to remind the reader of the change of a writing convention: in our previous papers (e.g. \cite{Guo3,GWYZ}), we used to employ $\tilde{I}_{A}$ for the equivalence class of $I_{A}$ for any given $A\in \mathcal{F}$, where $I_{A}$ stands for the characteristic function of $A$, namely $I_{A}(\omega)=1$ if $\omega\in A$ and $0$ otherwise, but in this paper we usually use $I_{a}$ for $\tilde{I}_{A}$, where $a$ is the equivalence class of $A$, since we want to state the definitions and results of this paper in the language of a Boolean algebra.

Let us end the introduction by recalling the notion of a regular $L^{0}(\mathcal{F},\mathbb{K})$-module (namely, a regular left module over the algebra $L^{0}(\mathcal{F},\mathbb{K})$) from \cite{Guo3,WGL} as follows. An $L^{0}(\mathcal{F},\mathbb{K})$-module $S$ is said to be regular if, for any two elements $x$ and $y$ in $S$, there is a partition $\{a_{n},n\in \mathbb{N}\}$ of unity in $B_{\mathcal{F}}$ such that $I_{a_{n}}x=I_{a_{n}}y$ for each $n\in \mathbb{N}$, then $x$ must equal to $y$ (equivalently, if, for any element $x$ in $S$, there is a partition $\{a_{n},n\in \mathbb{N}\}$ of unity in $B_{\mathcal{F}}$ such that $I_{a_{n}}x=\theta$ for each $n\in \mathbb{N}$, then it must hold that $x=\theta$ (the null element of $S$)). Throughout this paper, we always assume that all the $L^{0}(\mathcal{F},\mathbb{K})$-modules occurring in the sequel of this paper are regular, the restriction is not excessive since all random normed modules and random locally convex modules are regular.

%###################################################################################################
%###################################################################################################
\section{On $\sigma$-stability and $d$-$\sigma$-stability}
%###################################################################################################
%###################################################################################################

The aim of this section is to prove that the notion of $d$-$\sigma$-stability for a nonempty subset in a random metric space can be regarded as a special case of the notion of $\sigma$-stability for a nonempty subset in a random normed module, which will be achieved by isometrically embedding a random metric space into a properly constructed $\mathcal{T}_{\varepsilon,\lambda}$-complete random normed module. The main results in this section are Theorems \ref{thm.2.6}, \ref{thm.2.17}, \ref{thm2.18} and \ref{thm.2.21}. For this, we first introduce the relevant notions in random functional analysis as follows.

The original notions of a random metric space and a random normed space were introduced in \cite[Chapters 9 and 15]{SS} by defining the random distance between two points or the random norm of a vector as a nonnegative random variable defined on a probability space, the following slightly modified versions of them were given in \cite{Guo6a} (see \cite{Guo1,Guo2} for the reason why we gave such modified versions) by defining random distances or random norms as equivalence classes of nonnegative random variables or measurable functions in order to provide a convenient connection with the problems in functional analysis. Based on the modified version of a random normed space, the notion of a random normed module was introduced in \cite{Guo6aa,Guo1}, which leads to a series of subsequent developments of random functional analysis, see \cite{Guo3} for a historical survey. It should also be mentioned that the notion of an $RN$ module was independently introduced in \cite{HLR} (where it is called a randomly normed $L^0$-module) as a tool for the study of ultrapower of Lebesgue-Bocher function spaces.

\begin{defn}\label{def.2.1}
	An ordered pair $(S,d)$ is called a random metric space (briefly, an $RM$ space) with base $(\Omega,\mathcal{F},\mu)$ if $S$ is a nonempty set and $d$ is a mapping from $S\times S$ to $L_{+}^{0}(\mathcal{F})$ such that the following axioms are satisfied:
	\begin{itemize}
	\item [(1)] $d(x,y)=0$ iff $x=y$;
	\item [(2)] $d(x,y)=d(y,x)$ for any $x$ and $y$ in $S$;
	\item [(3)] $d(x,z)\leq d(x,y)+d(y,z)$ for any $x,y$ and $z$ in $S$.
	\end{itemize}
	As usual, $d$ is called the random metric on $S$.	
\end{defn}

For the study of the uniformity and the topology on an $RM$ space, a probability measure $P_{\mu}$ associated with the measure space $(\Omega,\mathcal{F},\mu)$ is defined as follows. When $(\Omega,\mathcal{F},\mu)$ is a finite measure space, $P_{\mu}(A)=\frac{\mu(A)}{\mu(\Omega)}$ for each $A\in \mathcal{F}$; when $(\Omega,\mathcal{F},\mu)$ is a $\sigma$-finite measure space, for example, let $\{A_{n},n\in \mathbb{N}\}$ be a countable partition of $\Omega$ to $\mathcal{F}$ such that $0<\mu(A_{n})<+\infty$ for each $n$ in $\mathbb{N}$, then $P_{\mu}(A)=\sum_{n=1}^{\infty}\frac{\mu(A\cap A_{n})}{2^{n}\mu(A_{n})}$ for each $A\in \mathcal{F}$.

The idea of introducing the $(\varepsilon,\lambda)$-uniformity (structure) inherits from that of introducing the $(\varepsilon,\lambda)$-uniformity for an abstract probabilistic metric space by Schweizer and Sklar in \cite{SS}, but the method of employing $P_{\mu}$ to introduce the $(\varepsilon,\lambda)$-uniformity and the $(\varepsilon,\lambda)$-topology for an $RM$ space with a $\sigma$-finite measure space $(\Omega,\mathcal{F},\mu)$ as base in Proposition \ref{pro.2.2} below belongs to \cite{Guo2}. The idea of introducing $L^{0}$-uniformity is very similar to that of introducing the locally $L^{0}$-convex topology by Filipovi\'{c}, et.al in \cite{FKV}.

\begin{prop}\label{pro.2.2}
	Let $(S,d)$ be an $RM$ space with base $(\Omega,\mathcal{F},\mu)$. Given $\varepsilon>0$ and $0<\lambda<1$, let $U(\varepsilon,\lambda)=\{(x,y)\in S\times S:P_{\mu}\{\omega\in \Omega:d(x,y)(\omega)<\varepsilon\}>1-\lambda\}$. Then $\mathcal{U}=\{U(\varepsilon,\lambda):\varepsilon>0,0<\lambda<1\}$ forms a base for some metrizable uniformity on $S$, called the $(\varepsilon,\lambda)$-uniformity for $S$ induced by $d$, and the topology induced by the $(\varepsilon,\lambda)$-uniformity is called the $(\varepsilon,\lambda)$-topology for $S$. Given $\varepsilon\in L_{++}^{0}(\mathcal{F})$, let $U'(\varepsilon)=\{(x,y)\in S\times S:d(x,y)<\varepsilon~\text{on}~\Omega\}$. Then $\mathcal{U}'=\{U'(\varepsilon):\varepsilon\in L_{++}^{0}(\mathcal{F})\}$ forms a base for some Hausdorff uniformity on $S$, called the $L^{0}$-uniformity for $S$ induced by $d$, and the topology induced by the $L^{0}$-uniformity is called the $L^{0}$-topology for $S$.
\end{prop}

In the sequel of this paper, for any $RM$ space $(S,d)$, we always use $\mathcal{U}_{\varepsilon,\lambda}$ and $\mathcal{T}_{\varepsilon,\lambda}$ for the $(\varepsilon,\lambda)$-uniformity and the $(\varepsilon,\lambda)$-topology induced by $d$, and $\mathcal{U}_{c}$ and $\mathcal{T}_{c}$ for the $L^{0}$-uniformity and the $L^{0}$-topology induced by $d$. We say that $(S,d)$ is $(\varepsilon,\lambda)$-complete ($L^{0}$-complete) if $S$ is complete with respect to $\mathcal{U}_{\varepsilon,\lambda}$ ($\mathcal{U}_{c}$).

\begin{rem}\label{rem.2.3}
	When $(\Omega,\mathcal{F},\mu)$ is a probability space, the $(\varepsilon,\lambda)$-uniformity and the $L^{0}$-uniformity were deeply studied in \cite{GWYZ}, at this time a sequence $\{x_{n},n\in \mathbb{N}\}$ in $(S,d)$ converges in $\mathcal{T}_{\varepsilon,\lambda}$ to $x$ in $S$ iff $\{d(x_{n},x),n\in \mathbb{N}\}$ converges in probability measure to $0$. When $(\Omega,\mathcal{F},\mu)$ ia a $\sigma$-finite measure space, a sequence $\{x_{n},n\in \mathbb{N}\}$ in $(S,d)$ converges in $\mathcal{T}_{\varepsilon,\lambda}$ to $x$ in $S$ iff $\{d(x_{n},x),n\in \mathbb{N}\}$ converges locally in measure $\mu$ to $0$, namely, for each $A\in \mathcal{F}$ with $0<\mu(A)<+\infty$, $\{d(x_{n},x),n\in \mathbb{N}\}$ converges in measure $\mu$ to $0$ on $A$, it is in order to obtain the topology of convergence locally in measure that we employ $P_{\mu}$ in Proposition \ref{pro.2.2} as Guo earlier did in \cite{Guo2}. As to the problem of why the topology of convergence locally in measure is important, please refer to Remark \ref{rem.2.5} of this paper. From a formal perspective, $L^{0}$-uniformity is very similar to the uniformity of an ordinary metric space, it is, however, not metrizable in general even when $(\Omega,\mathcal{F},\mu)$ is a finite measure space. Besides, it is also interesting to consider the local $L^{0}$-uniformity for an $RM$ space $(S,d)$ with base $(\Omega,\mathcal{F},\mu)$: given $\varepsilon\in L_{++}^{0}(\mathcal{F})$ and $A\in \mathcal{F}$ with $0<\mu(A)<+\infty$, let $U(A,\varepsilon)=\{(x,y)\in S\times S:d(x,y)<\varepsilon~\text{on}~A\}$. Then $\mathcal{U}_{loc}=\{U(A,\varepsilon):A\in \mathcal{F} ~\text{with}~ 0<\mu(A)<+\infty,\varepsilon\in L_{++}^{0}(\mathcal{F})\}$ is also a base for some Hausdorff uniformity on $S$, called the local $L^{0}$-uniformity for $S$, which is weaker than the $L^{0}$-uniformity but stronger than the $(\varepsilon,\lambda)$-uniformity. We only use the $L^{0}$-uniformity in this paper in order to be in accordance with the relevant study in \cite{DJKK}.
\end{rem}

\begin{defn}\label{def.2.4}
	An ordered pair $(S,\|\cdot\|)$ is called a random normed space (briefly, an $RN$ space) over the scalar field $\mathbb{K}$ with base $(\Omega,\mathcal{F},\mu)$ if $S$ is a linear space over $\mathbb{K}$ and $\|\cdot\|$ is a mapping from $S$ to $L_{+}^{0}(\mathcal{F})$ such that the following axioms are satisfied:\\
	(RN-1) $\|x\|=0$ iff $x=\theta$;\\
	(RN-2) $\|\alpha x\|=|\alpha|\|x\|$ for any $x$ in $S$ and $\alpha$ in $\mathbb{K}$;\\
	(RN-3) $\|x+y\|\leq \|x\|+\|y\|$ for any $x$ and $y$ in $S$.	\\
%	\begin{itemize}
%		\item [(RN-1)] $\|x\|=0$ iff $x=\theta$;
%		\item [(RN-2)] $\|\alpha x\|=|\alpha|\|x\|$ for any $x$ in $S$ and $\alpha$ in $\mathbb{K}$;
%		\item [(RN-3)]$\|x+y\|\leq \|x\|+\|y\|$ for any $x$ and $y$ in $S$.	
%	\end{itemize}
	As usual, $\|\cdot\|$ is called the random norm on $S$. In addition, if $S$ is an $L^{0}(\mathcal{F},\mathbb{K})$-module such that the following axiom is also satisfied:\\
	(RNM-1) $\|\xi x\|=|\xi|\|x\|$ for any $\xi$ in $L^{0}(\mathcal{F},\mathbb{K})$ and any $x$ in $S$,\\
%	\begin{itemize}
%		\item [(RNM-1)] $\|\xi x\|=|\xi|\|x\|$ for any $\xi$ in $L^{0}(\mathcal{F},\mathbb{K})$ and any $x$ in $S$,
%	\end{itemize}
	then the $RN$ space $(S,\|\cdot\|)$ is called a random normed module (briefly, an $RN$ module) over $\mathbb{K}$ with base $(\Omega,\mathcal{F},\mu)$, in which case the random norm (namely, satisfying (RNM-1) is also called $L^{0}$-norm). Naturally, if a mapping $\|\cdot\|$ from $S$ to $L_{+}^{0}(\mathcal{F})$ only satisfies (RN-3) and (RNM-1), then it is called an $L^{0}$-seminorm on the $L^{0}(\mathcal{F},\mathbb{K})$-module $S$.
\end{defn}

The algebra $L^{0}(\mathcal{F},\mathbb{K})$ is a simplest $RN$ module over $\mathbb{K}$ with base $(\Omega,\mathcal{F},\mu)$ when it is endowed with the random norm defined by $\|\xi\|=|\xi|$ for any $\xi$ in $L^{0}(\mathcal{F},\mathbb{K})$. When $L^{0}(\mathcal{F},\mathbb{K})$ is used as an $RN$ module in random functional analysis, its random norm or $L^{0}$-norm is always assumed to be $|\cdot|$ (namely, the absolute value mapping). For any $RN$ space $(S,\|\cdot\|)$ over $\mathbb{K}$ with base $(\Omega,\mathcal{F},\mu)$, $d:S\times S\rightarrow L_{+}^{0}(\mathcal{F})$ defined by $d(x,y)=\|x-y\|$ for any $x$ and $y$ in $S$, is clearly a random metric on $S$, the $(\varepsilon,\lambda)$-topology induced by $d$ is, as usual, also called the  $(\varepsilon,\lambda)$-topology for the $RN$ space, and it is easy to check that the $(\varepsilon,\lambda)$-topology $\mathcal{T}_{\varepsilon,\lambda}$ on $(S,\|\cdot\|)$ is always a metrizable linear topology, in particular $(L^{0}(\mathcal{F},\mathbb{K}),\mathcal{T}_{\varepsilon,\lambda})$ is a metrizable topological algebra over $\mathbb{K}$ (namely, the algebraic multiplication operation is also jointly continuous). Furthermore, when $(S,\|\cdot\|)$ is an $RN$ module over $\mathbb{K}$ with base $(\Omega,\mathcal{F},\mu)$, $(S,\mathcal{T}_{\varepsilon,\lambda})$ is a metrizable topological module over the topological algebra $(L^{0}(\mathcal{F},\mathbb{K}),\mathcal{T}_{\varepsilon,\lambda})$, namely, the module multiplication $\cdot:L^{0}(\mathcal{F},\mathbb{K})\times S\rightarrow S$ is jointly continuous under the respective $(\varepsilon,\lambda)$-topologies of $L^{0}(\mathcal{F},\mathbb{K})$ and $S$. Similarly, the $L^{0}$-topology $\mathcal{T}_{c}$ induced by the corresponding random metric $d$ is just the locally $L^{0}$-convex topology introduced in \cite{FKV} for an $RN$ module $(S,\|\cdot\|)$, but $\mathcal{T}_{c}$ is no longer a linear topology in general even when its base space $(\Omega,\mathcal{F},\mu)$ is a finite measure space since the scalar multiplication is nolonger continuous, for example, $(L^{0}(\mathcal{F},\mathbb{K}), \mathcal{T}_{c})$ is merely a topological ring (see \cite{FKV} for details), $(S,\mathcal{T}_{c})$ is also merely a topological module over the topological ring $(L^{0}(\mathcal{F},\mathbb{K}), \mathcal{T}_{c})$ for any $RN$ module over $\mathbb{K}$ with base  $(\Omega,\mathcal{F},\mu)$.

\begin{rem}\label{rem.2.5}
	Since $L^{0}(\mathcal{F},\mathbb{K})$ is an algebra with the unit element $1$ (namely the equivalence class determined by the constant function with value $1$ on $(\Omega,\mathcal{F},\mu)$), (RNM-1) in Definition \ref{def.2.4} is compatible with (RN-2), and thus (RNM-1) strengthens (RN-2) in the case of an $RN$ module, then in this case we can replace (RN-2) with (RNM-1). By the way, we should point out that although the topology of convergence in measure $\mu$ is metrizable, it is not a linear topology when $\mu$ is $\sigma$-finite but not finite, which is the reason why we employ $P_{\mu}$ in introducing the $(\varepsilon,\lambda)$-uniformity in Proposition \ref{pro.2.2}.
\end{rem}

\begin{thm}\label{thm.2.6}
	Let $(S,d)$ be an $RM$ space with base $(\Omega,\mathcal{F},\mu)$ and $\mathcal{U}_{b}(S,L^{0}(\mathcal{F}))$ be the set of all the uniformly continuous mappings $f$ from $(S,\mathcal{U}_{\varepsilon,\lambda})$ to $(L^{0}(\mathcal{F},$
$\mathbb{K}),\mathcal{U}_{\varepsilon,\lambda})$ such that $\bigvee\{|f(p)|:p\in S\}\in L_{+}^{0}(\mathcal{F})$ (namely, $f$ is a.e. bounded). Then we have the following statements:
	\begin{itemize}
		\item [(1)] $\mathcal{U}_{b}(S,L^{0}(\mathcal{F}))$ becomes a $\mathcal{T}_{\varepsilon,\lambda}$-complete $RN$ module over $\mathbb{R}$ with base $(\Omega,\mathcal{F},\mu)$ under the ordinary addition operation $+$ defined by $(f_{1}+f_{2})(p)=f_{1}(p)+f_{2}(p)$ for any $f_{1}$ and $f_{2}$ in $\mathcal{U}_{b}(S,L^{0}(\mathcal{F}))$ and any $p\in S$, and the module multiplication $\cdot$ defined by $(\xi f)(p)=\xi\cdot(f(p))$ for any $(\xi,f)\in L^{0}(\mathcal{F})\times \mathcal{U}_{b}(S,L^{0}(\mathcal{F}))$ and any $p\in S$, while $\mathcal{U}_{b}(S,L^{0}(\mathcal{F}))$ is endowed with the $L^{0}$-norm $\|\cdot\|$ defined by $\|f\|=\bigvee\{|f(p)|:p\in S\}$ for any $f\in \mathcal{U}_{b}(S,L^{0}(\mathcal{F}))$.
		\item [(2)]	Define $T:S\rightarrow \mathcal{U}_{b}(S,L^{0}(\mathcal{F}))$ as follows: arbitrarily fix an element $p_{0}\in S$, $T(p)=d(p,\cdot)-d(p_{0},\cdot)$ for any $p\in S$, then $T$ is an isometric embedding of $(S,d)$ into $(\mathcal{U}_{b}(S,L^{0}(\mathcal{F})),\|\cdot\|)$, namely $\|T(p)-T(q)\|=d(p,q)$ for any $p$ and $q$ in $S$.
	\end{itemize}
\end{thm}
\begin{proof}	
	(1) Since $\mu$ and $P_{\mu}$ are equivalent to each other, namely for any $A\in \mathcal{F},~\mu(A)=0$ iff $P_{\mu}(A)=0$, and the proof is only related to equivalence of measures, we can, without loss of generality, assume that $\mu$ is a probability measure, at this time $\mu=P_{\mu}$. Since $(\mathcal{U}_{b}(S,L^{0}(\mathcal{F})),\|\cdot\|)$ is clearly an $RN$ module, we only need to prove that it is $\mathcal{T}_{\varepsilon,\lambda}$-complete.
	
	Now, let $\{f_{n},n\in \mathbb{N}\}$ be a $\mathcal{T}_{\varepsilon,\lambda}$-Cauchy sequence in $(\mathcal{U}_{b}(S,L^{0}(\mathcal{F})),\|\cdot\|)$, since $|f_{n}(p)-f_{m}(p)|\leq \|f_{n}-f_{m}\|$ for any $p\in S$ and any $n$ and $m$ in $\mathbb{N}$, then it is obvious that $\{f_{n}(p),n\in \mathbb{N}\}$ is also a $\mathcal{T}_{\varepsilon,\lambda}$-Cauchy sequence in $L^{0}(\mathcal{F})$ for each $p$ in $S$, further by $\mathcal{T}_{\varepsilon,\lambda}$-completeness of $L^{0}(\mathcal{F})$, $f:S\rightarrow L^{0}(\mathcal{F})$ defined by $f(p)=\mathcal{T}_{\varepsilon,\lambda}$-limit of $\{f_{n}(p),n\in \mathbb{N}\}$ for each $p\in S$, is well defined. We continue to prove that $f\in \mathcal{U}_{b}(S,L^{0}(\mathcal{F}))$ and $\{\|f_{n}-f\|,n\in \mathbb{N}\}$ converges in probability measure $\mu$ to $0$ as follows.
	
	Since $\{f_{n},n\in \mathbb{N}\}$ is $\mathcal{T}_{\varepsilon,\lambda}$-Cauchy, there exists a subsequence $\{f_{n_{k}},k\in \mathbb{N}\}$ is a.s. Cauchy, namely $\lim_{k,l\rightarrow +\infty}\|f_{n_{k}}-f_{n_{l}}\|(\omega)=0~a.e.$, then it follows immediately from $|\|f_{n_{k}}\|-\|f_{n_{l}}\||\leq \|f_{n_{k}}-f_{n_{l}}\|$ that $\vee_{k\geq 1}\|f_{n_{k}}\|\in L_{+}^{0}(\mathcal{F})$, it is also easy to see, from $|f_{n_{k}}(p)-f_{n_{l}}(p)|\leq \|f_{n_{k}}-f_{n_{l}}\|$ for any $p\in S$, that $\{|f_{n_{k}}(p)-f(p)|,k\in \mathbb{N}\}$ converges a.s. to $0$ for any $p\in S$, which implies that $|f(p)|=a.e.-\lim_{k\rightarrow +\infty}|f_{n_{k}}(p)|\leq \vee_{k\geq 1}\|f_{n_{k}}\|$ for any $p\in S$, namely $\bigvee\{|f(p)|:p\in S\}\in L_{+}^{0}(\mathcal{F})$. It remains to prove that $f$ is uniformly continuous from $(S,\mathcal{U}_{\varepsilon,\lambda})$ to $(L^{0}(\mathcal{F}),\mathcal{U}_{\varepsilon,\lambda})$ and $\{f_{n_{k}},k\in \mathbb{N}\}$ converges in $\mathcal{T}_{\varepsilon,\lambda}$ to $f$ as follows.
	
	Since $\{f_{n_{k}},k\in \mathbb{N}\}$ is a.e. Cauchy, by Egoroff's theorem $\{f_{n_{k}},k\in \mathbb{N}\}$ is also almost uniformly (or $\mu$-uniformly) Cauchy, namely, for any $\varepsilon>0$ and $0<\lambda<1$, there exist $N(\varepsilon,\lambda)\in \mathbb{N}$ and $\Omega(\varepsilon,\lambda)\in \mathcal{F}$ such that $\mu(\Omega(\varepsilon,\lambda))>1-\lambda$ and $\|f_{n_{k}}-f_{n_{l}}\|(\omega)\leq \varepsilon$ on $\Omega(\varepsilon,\lambda)$ whenever $k,l\geq N(\varepsilon,\lambda)$. From $|f_{n_{k}}(p)-f_{n_{l}})(p)|\leq \|f_{n_{k}}-f_{n_{l}}\|$ for any $p\in S$, one can have that $|f_{n_{k}}(p)-f_{n_{l}})(p)|\leq \varepsilon$ on $\Omega(\varepsilon,\lambda)$ for any $k,l\geq N(\varepsilon,\lambda)$ and any $p\in S$. Now, letting $l\rightarrow +\infty$ yields that $|f_{n_{k}}(p)-f(p)|\leq \varepsilon$ on $\Omega(\varepsilon,\lambda)$ for any $k\geq N(\varepsilon,\lambda)$ and any $p\in S$, which further implies that $\xi_{k}:=\bigvee\{|f_{n_{k}}(p)-f(p)|:p\in S\}\leq \varepsilon$ on $\Omega(\varepsilon,\lambda)$ whenever $k\geq N(\varepsilon,\lambda)$ (please note: once we can prove that $f$ is uniformly continuous from $(S,\mathcal{U}_{\varepsilon,\lambda})$ to $(L^{0}(\mathcal{F}),\mathcal{U}_{\varepsilon,\lambda})$, then $\xi_{k}$ is exactly $\|f_{n_{k}}-f\|$ for any $k\in \mathbb{N}$). Thus, we have proved that $\{\xi_{k},k\in \mathbb{N}\}$ almost uniformly converges to $0$. Finally, we prove that $f$ is uniformly continuous from $(S,\mathcal{U}_{\varepsilon,\lambda})$ to $(L^{0}(\mathcal{F}),\mathcal{U}_{\varepsilon,\lambda})$ as follows.
	
	Since  $\{\xi_{k},k\in \mathbb{N}\}$ almost uniformly converges to $0$, it, of course, also converges in probability measure $\mu$ to $0$, then for any given $\varepsilon_{1}>0$ and $0<\lambda_{1}<1$ there exists  $k_{0}:=N(\varepsilon_{1},\lambda_{1})\in \mathbb{N}$ such that $\mu\{\omega\in \Omega:\xi_{k}(\omega)<\frac{\varepsilon_{1}}{4}\}>1-\frac{\lambda_{1}}{2}$ when $k\geq k_{0}$. Again since $f_{n_{k_{0}}}$ is uniformly continuous from $(S,\mathcal{U}_{\varepsilon,\lambda})$ to $(L^{0}(\mathcal{F}),\mathcal{U}_{\varepsilon,\lambda})$, there exist $\varepsilon_{2}>0$ and $0<\lambda_{2}<1$ such that $\mu\{\omega\in \Omega:|f_{n_{k_{0}}}(p)-f_{n_{k_{0}}}(q)|(\omega)<\frac{\varepsilon_{1}}{2}\}>1-\frac{\lambda_{1}}{2}$ whenever $\mu\{\omega\in \Omega:d(p,q)(\omega)<\varepsilon_{2}\}>1-\lambda_{2}$. Since $|f(p)-f(q)|\leq |f(p)-f_{n_{k_{0}}}(p)|+|f_{n_{k_{0}}}(p)-f_{n_{k_{0}}}(q)|+|f_{n_{k_{0}}}(q)-f(q)|\leq|f_{n_{k_{0}}}(p)-f_{n_{k_{0}}}(q)|+2\xi_{k_{0}} $ for any $(p,q)\in S\times S$, then $\mu\{\omega\in \Omega:|f(p)-f(q)|(\omega)<\varepsilon_{1}\}\geq \mu\{\omega\in \Omega:|f_{n_{k_{0}}}(p)-f_{n_{k_{0}}}(q)|(\omega)<\frac{\varepsilon_{1}}{2}\}+\mu\{\omega\in \Omega:\xi_{k_{0}}(\omega)<\frac{\varepsilon_{1}}{4}\}-1>2(1-\frac{\lambda_{1}}{2})-1=1-\lambda_{1}$ whenever $\mu\{\omega\in \Omega:d(p,q)(\omega)<\varepsilon_{2}\}>1-\lambda_{2}$. Thus, $f$ is uniformly continuous, namely $f\in \mathcal{U}_{b}(S,L^{0}(\mathcal{F}))$, at this time $\{f_{n_{k}},k\in \mathbb{N}\}$ also converges in $\mathcal{T}_{\varepsilon,\lambda}$ to $f$ since $\{\xi_{k}=\|f_{n_{k}}-f\|,k\in \mathbb{N}\}$ converges in probability measure $\mu$ to $0$, then $\{f_{n},n\in \mathbb{N}\}$ also converges in $\mathcal{T}_{\varepsilon,\lambda}$ to $f$ since it is $\mathcal{T}_{\varepsilon,\lambda}$-Cauchy.
	
	(2) When the probability space $(\Omega,\mathcal{F},\mu)$ is trivial, namely $\mathcal{F}=\{\Omega,\emptyset\}$, the $RM$ space $(S,d)$ reduces to an ordinary metric space, $\mathcal{U}_{b}(S,L^{0}(\mathcal{F}))$ to the Banach space of bounded real-valued uniformly continuous functions on $S$, and (2) to \cite[Embedding Lemma 3.23, p.84]{AB}. The proof of (2) in the general case is completely similar to the proof of Embedding Lemma 3.23 of \cite{AB}, so is omitted.
\end{proof}

\begin{rem}\label{rem.2.7}
	Let $(S,d)$ be an RM space with base $(\Omega,\mathcal{F},\mu)$ and $C_{b}(S,L^0(\mathcal{F}))$ be the set of all the continuous mappings $f$ from $(S,\mathcal{T}_{\varepsilon,\lambda})$ to $(L^0(\mathcal{F}),\mathcal{T}_{\varepsilon,\lambda} )$ such that $\bigvee\{|f(p)|:p\in S\}\in L^0_{+}(\mathcal{F})$, then, similarly to (1) of Theorem \ref{thm.2.6}, one can see that $(C_{b}(S,L^0(\mathcal{F})), \| \cdot\|)$ becomes a $\mathcal{T}_{\varepsilon,\lambda}$-complete RN module over $\mathbb{R}$ with base $(\Omega,\mathcal{F},\mu)$ and $\mathcal{U}_{b}(S,L^0(\mathcal{F}))$ is a $\mathcal{T}_{\varepsilon,\lambda}$-closed submodule of $C_{b}(S,L^0(\mathcal{F}))$.
\end{rem}

\begin{rem}\label{rem.2.8}
	Let $(S,d)$ be an RM space with base $(\Omega,\mathcal{F},\mu)$. Arbitrarily choose a countable partition $\{A_{i},i\in \mathbb{N}\}$ of $\Omega$ to $\mathcal{F}$ such that $0<\mu(A_i)<+\infty$ for each $i\in \mathbb{N}$, then it is easy to see that the $(\varepsilon,\lambda)$-uniformity
	$\mathcal{U}_{\varepsilon,\lambda}$ on $(S,d)$ is generated by the metric $D$ on $S$ defined by $D(x,y)=\sum_{i\geq 1}\frac{1}{2^i\mu(A_i)}\int_{A_i}(d(x,y)\wedge 1)d\mu$ for any $x$ and $y$ in $S$. Similarly, when $(S,\| \cdot\|)$ is an RN space over $\mathbb{K}$ with base $(\Omega,\mathcal{F},\mu)$, its metrizable $(\varepsilon,\lambda)$-linear topology $\mathcal{T}_{\varepsilon,\lambda}$ can be generated by the quasinorm $\normmm{\cdot}$ on $S$ defined by $\normmm{x}=\sum_{i\geq 1}\frac{1}{2^i\mu(A_i)}\int_{A_i}(\|x\|\wedge 1)d\mu$ for each $x\in S$.
\end{rem}

Since $L^0(\mathcal{F})$ is both $\mathcal{T}_{\varepsilon,\lambda}$-complete and $\mathcal{T}_{c}$-complete as an RN module, just as every metric (or, normed) space admits a completion, the completion proposition below is also obvious.

\begin{prop}\label{pro.2.9}
	For every RM space $(S,d)$ with base $(\Omega,\mathcal{F},\mu)$, there exists an $(\varepsilon,\lambda)$-complete RM space $(\tilde{S},\tilde{d})$ with base $(\Omega,\mathcal{F},\mu)$ and a random-metric-preserving embedding $\varphi$ from $S$ to $\tilde{S}$ such that $\varphi(S)$ is $\mathcal{T}_{\varepsilon,\lambda}$-dense in $\tilde{S}$, such $(\tilde{S},\tilde{d})$ is unique in the sense of an isometric isomorphism with respect to random metric, called the completion of $(S,d)$, and $\tilde{d}$ is still denoted by $d$. For every RN space (module) $(S,\|\cdot\|)$ over $\mathbb{K}$ with base $(\Omega,\mathcal{F},\mu)$, there exists a $\mathcal{T}_{\varepsilon,\lambda}$-complete RN space (module)
	$(\tilde{S},\| \cdot\|^{\tilde{}})$ and a random-norm-preserving embedding $\varphi$ from $S$ into $\tilde{S}$ such that $\varphi$ is a linear operator (or a module homomorphism) and $\varphi(S)$ is $\mathcal{T}_{\varepsilon,\lambda}$-dense in $\tilde{S}$, such $(\tilde{S},\| \cdot\|^{\tilde{}})$ is unique in the sense of isometric isomorphism, called the completion of $(S,\|\cdot\|)$ and $\| \cdot\|^{\tilde{}}$ is still denoted by $\|\cdot\|$.
\end{prop}

Definition \ref{def.2.10} below is taken from \cite[Definition 3.1]{Guo3}.

\begin{defn}\label{def.2.10}
	Let $S$ be an $L^0(\mathcal{F},\mathbb{K})$-module and $G$ be a nonempty subset of $S$. $G$ is said to be stable if $I_{a}x+I_{a^c}y\in G$ for any $x$ and $y$ in $G$ and any $a\in B_{\mathcal{F}}$ (let us recall: for any $a=[A]$ with $A\in \mathcal{F}$, $I_{a}$ stands for $\tilde{I}_{A}$ , and $I_{a^c}=\tilde{I}_{A^c}$, where $A^c=\Omega\backslash A$). $G$ is said to be $\sigma$-stable (or to have the countable concatenation property in terms of \cite{Guo3}) if, for any sequence $\{x_n,n\in \mathbb{N}\}$ in $G$ and any partition $\{a_n,n\in \mathbb{N}\}$ of unity in $B_{\mathcal{F}}$, there is $x$ in $G$ such that $I_{a_n}x=I_{a_n}x_n$ for each $n\in \mathbb{N}$ ($x$ is unique since $S$ is assumed to be regular, usually denoted by $\sum_{n}I_{a_n}x_n$, called the countable concatenation of $\{x_n,n\in \mathbb{N}\}$ along $\{a_n,n\in \mathbb{N}\}$). By the way, if $G$ is $\sigma$-stable and  $H$ is a nonempty subset of $G$, then $\sigma(H):=\{\sum_{n}I_{a_n}h_n: \{h_n,n\in \mathbb{N}\} ~$is a sequence in$~ H ~$and$~ \{a_n,n\in \mathbb{N}\} ~$is a partition of unity in $B_{\mathcal{F}}\}$ is called the $\sigma$-stable hull of $H$.
\end{defn}

It is obvious that any $L^0(\mathcal{F},\mathbb{K})$-module $S$ and any $L^0$-convex subset of $S$ (see Section 3 below for the notion of an $L^0$-convex set) are stable. Following are some important examples of $\sigma$-stable sets.

\begin{ex}\label{exm.2.11}
	Let $(S,\mathcal{P})$ be a random locally convex module over $\mathbb{K}$ with base $(\Omega,\mathcal{F},\mu)$, namely $\mathcal{P}$ is a family of $L^0$-seminorms on $S$ such that $\bigvee\{\|x\|: \|\cdot\| \in \mathcal{P}\}=0$ iff $x=\theta$. Let $\mathcal{P}_{f}$ be the family of finite subsets of $\mathcal{P}$ and $\| \cdot\|_{Q}$ stand for the $L^0$-seminorm defined by $\| x\|_{Q}=\bigvee\{\|x\|: \|\cdot\| \in Q\}$ for any $x\in S$, where $Q\in \mathcal{P}_{f}$. Given $\varepsilon>0$, $0<\lambda<1$ and $Q\in \mathcal{P}_{f}$, let $U_{\theta}(Q,\varepsilon,\lambda)=\{x\in S: P_{\mu}\{\omega\in \Omega: \| x\|_{Q}<\varepsilon \}>1-\lambda \}$. Then $\mathcal{U}_{\theta}=\{U_{\theta}(Q,\varepsilon,\lambda): Q\in \mathcal{P}_{f},\varepsilon>0,0<\lambda<1\}$ forms a local base for some Hausdorff linear topology on $S$, called the $(\varepsilon,\lambda)$-topology for $S$ and denoted by $\mathcal{T}_{\varepsilon,\lambda}$. Furthermore, $(S,\mathcal{T}_{\varepsilon,\lambda})$ is a topological module over the topological algebra $(L^0(\mathcal{F},\mathbb{K}),\mathcal{T}_{\varepsilon,\lambda})$. Given $Q\in \mathcal{P}_{f}$ and $\varepsilon \in L^0_{++}(\mathcal{F})$, let $V_{\theta}(Q,\varepsilon)=\{x\in S: \| x\|_{Q}<\varepsilon ~$on$~\Omega\}$. Then $\mathcal{V}_{\theta}=\{V_{\theta}(Q,\varepsilon):Q\in \mathcal{P}_{f} ~$and$~\varepsilon \in L^0_{++}(\mathcal{F})\}$ forms a locally base for some Hausdorff locally $L^0$-convex topology on $S$ in the sense of \cite{FKV}, denoted by $\mathcal{T}_{c}$, moreover, $(S,\mathcal{T}_{c})$ is a topological module over the topological ring $(L^0(\mathcal{F},\mathbb{K}),\mathcal{T}_{c})$. Although $\mathcal{T}_{c}$ is much stronger that $\mathcal{T}_{\varepsilon,\lambda}$, it is proved in \cite{Guo3} that $G^{-}_{c}=G^{-}_{\varepsilon,\lambda}$ for any $\sigma$-stable subset $G$ of $S$, where $G^{-}_{c}$ and $G^{-}_{\varepsilon,\lambda}$ stand for the $\mathcal{T}_{c}$-closure and $\mathcal{T}_{\varepsilon,\lambda}$-closure of $G$, respectively. If $(S,\mathcal{P})$ is $\mathcal{T}_{\varepsilon,\lambda}$-complete, then $S$ is $\sigma$-stable: for any sequence $\{x_n,n\in \mathbb{N}\}$ and a partition $\{a_n,n\in \mathbb{N}\}$ of unity in $B_{\mathcal{F}}$, it is easy to check that $\{\sum_{k=1}^{n}I_{a_k}x_k,n\in \mathbb{N} \}$ is $\mathcal{T}_{\varepsilon,\lambda}$-Cauchy, and hence convergent to some $x\in S$, then $x$ just satisfies $I_{a_n}x=I_{a_n}x_n$ for each $n\in \mathbb{N}$. Similarly, a $\mathcal{T}_{\varepsilon,\lambda}$-complete stable subset of $S$ is also $\sigma$-stable. If $(S,\mathcal{P})$ is $\mathcal{T}_{c}$-complete, $S$ is, however, not necessarily $\sigma$-stable. It is interesting that, as Theorem 3.18 of \cite{Guo3} shows, $S$ is
	$\mathcal{T}_{\varepsilon,\lambda}$-complete iff $S$ is both $\mathcal{T}_{c}$-complete and  $\sigma$-stable. Just as pointed out in \cite{GWCXY}, for a stable subset $G$ of $S$, we also have that $G$ is $\mathcal{T}_{\varepsilon,\lambda}$-complete iff $G$ is both $\mathcal{T}_{c}$-complete and  $\sigma$-stable.
\end{ex}

\begin{ex}\label{exm.2.12}
	Let $(S,\|\cdot\|)$  be an RN space over $\mathbb{K}$ with base $(\Omega,\mathcal{F},\mu)$. A linear operator $f:S\rightarrow L^0(\mathcal{F},\mathbb{K})$ is called an a.e. bounded random linear functional on $S$ if there exists some $\xi \in L^0_{+}(\mathcal{F})$ such that $|f(x)|\leq \xi \|x\|$ for any $x\in S$. Denote by $S^*$ the set of a.e. bounded random linear functionals on $S$, $S^*$ is, clearly, an $L^0(\mathcal{F},\mathbb{K})$-module endowed with the module multiplication $(\xi f)(x)=\xi \cdot (f(x))$ for any $\xi \in L^0(\mathcal{F})$, any $f\in S^*$ and any $x\in S$, further $S^*$ automatically forms an RN module over  $\mathbb{K}$ with base $(\Omega,\mathcal{F},\mu)$ when it is endowed with the $L^0$-norm $\|f\|=\bigwedge\{\xi \in L^0_{+}(\mathcal{F}):\xi ~$is such that$~  |f(x)|\leq \xi \|x\|~$ for any$~ x\in S\}$ for any $f\in S^*$, called the random conjugate space of $(S,\|\cdot\|)$. The notion of an a.e. bounded random linear functional was introduced in \cite{Guo6a}, where the Hahn-Banach theorem was also proved: let $M\subset S$ be a subspace and $f:M\rightarrow L^{0}(\mathcal{F},\mathbb{K})$ be an a.e. bounded random linear functional, then there exists $F\in S^{*}$ such that $F|_{M}=f$ and $\|F\|=\|f\|$, see e.g. \cite{Guo3,Guo7} for details. A basic fact is that $f\in S^*$ iff $f$ is a continuous module homomorphism from $(S,\mathcal{T}_{\varepsilon,\lambda})$ to $(L^0(\mathcal{F},\mathbb{K}),\mathcal{T}_{\varepsilon,\lambda})$ when $S$ is an RN module, and at this time $\|f\|=\bigvee\{|f(x)|:x\in S(1) \}$, where $S(1)=\{x\in S: \|x\|\leq 1\}$. These observations are our initial motivation of introducing the notion of an RN module, see \cite{Guo6aa,Guo1,Guo2} for details. Similarly, for a random locally convex module $(S,\mathcal{P})$ over $\mathbb{K}$ with base $(\Omega,\mathcal{F},\mu)$, $S^*_{\varepsilon,\lambda}:=$ the $L^0(\mathcal{F},\mathbb{K})$-module of continuous module  homomorphisms from $(S,\mathcal{T}_{\varepsilon,\lambda})$ to $(L^0(\mathcal{F},\mathbb{K}),\mathcal{T}_{\varepsilon,\lambda})$, is called the random conjugate space of $(S,\mathcal{P})$ under $\mathcal{T}_{\varepsilon,\lambda}$, while $S^*_{c}:=$ the $L^0(\mathcal{F},\mathbb{K})$-module of continuous module  homomorphisms from $(S,\mathcal{T}_{c})$ to $(L^0(\mathcal{F},\mathbb{K}),\mathcal{T}_{c})$, is called the random conjugate space of $(S,\mathcal{P})$ under $\mathcal{T}_{c}$. It is proved in \cite{Guo3} that $S^*_{c}\subset S^*_{\varepsilon,\lambda}$ for any random locally convex module $(S,\mathcal{P})$, in particular $S^*_{c}= S^*_{\varepsilon,\lambda}$ when $S$ is an RN module. It is also observed that $S^*_{\varepsilon,\lambda}$ is always $\sigma$-stable, but $S^*_{c}$ is not necessarily $\sigma$-stable for a random locally convex module $(S,\mathcal{P})$, in particular it is subsequently found in \cite{GZZ1} that $S^*_{\varepsilon,\lambda}=\sigma(S^*_{c})$ for any random locally convex module $(S,\mathcal{P})$.
\end{ex}

\begin{ex}\label{exm.2.13}
	Let $(\Omega,\mathcal{E},P)$ be a probability space, $\mathcal{F}$ a $\sigma$-subalgebra of $\mathcal{E}$, $L^p(\mathcal{E})$ the classical $L^p$-space for any $p\in [1,+\infty]$ and $L^p_{\mathcal{F}}(\mathcal{E})$ the $L^0(\mathcal{F})$-module generated by $L^p(\mathcal{E})$. Then $L^p_{\mathcal{F}}(\mathcal{E})$ is a $\mathcal{T}_{\varepsilon,\lambda}$-complete RN module over $\mathbb{R}$ with base $(\Omega,\mathcal{F},P)$ and hence also $\sigma$-stable when it is endowed with the $L^0$-norm $\normmm{\cdot}_{p}$ (also called the condition $L^p$-norm) defined as follows, for any $x\in L^p_{\mathcal{F}}(\mathcal{E})$,
	$$ \normmm{x}_{p}=\left\{
	\begin{aligned}
		& E[|x|^p|\mathcal{F}]^{\frac{1}{p}}, & when~ p \in [1,+\infty); \\
		& \bigwedge\{\xi\in L^0_{+}(\mathcal{F}):|x|\leq \xi \}  , & when~ p=+\infty.
	\end{aligned}
	\right.
	$$
	The works \cite{FKV} and \cite{GZZ2} have showed that $L^p_{\mathcal{F}}(\mathcal{E})$ can be chosen as the model space universally suitable for the study of condition risk measures.
\end{ex}

The three examples have exhibited the fundamental importance of the notion of $\sigma$-stability in random functional analysis and its applications, we naturally wish a notion similar to $\sigma$-stability for a nonempty subset in an RM space, it is easy to observe that for a sequence $\{x_n,n\in \mathbb{N}\}$ in an RN module $(S,\|\cdot\|)$ over $\mathbb{K}$ with base $(\Omega,\mathcal{F},\mu)$, a partition $\{a_n,n\in \mathbb{N}\}$ of unity in $B_{\mathcal{F}}$ and an element $x$ in $S$, $I_{a_n}x=I_{a_n}x_n$ for each $n\in \mathbb{N}$ iff $I_{a_n}\|x-x_n\|=0$ for each $n\in \mathbb{N}$, which leads the authors in \cite{JKZ} and \cite{GWYZ} to the following notion of $d$-$\sigma$-stability (the original manuscript of \cite{JKZ} was earlier announced in arXiv than that of \cite{GWYZ}).

\begin{defn}\label{def.2.14}
	Let $(S,d)$  be an RM space with base $(\Omega,\mathcal{F},\mu)$ and $G$ be a nonempty subset of $S$. $G$ is said to be $d$-stable if, for any $x$ and $y\in G$ and any $a\in B_{\mathcal{F}}$, there exists $z\in G$ such that $I_{a}d(z,x)=0$ and $I_{a^c}d(z,y)=0$ (such an element $z$ is unique by the triangle inequality (RM-3) and usually denoted by $I_{a}x+I_{a^c}y$). $G$ is said to be $d$-$\sigma$-stable if, for any sequence $\{x_n,n\in \mathbb{N}\}$ in $G$ and any partition $\{a_n,n\in \mathbb{N}\}$ of unity in $B_{\mathcal{F}}$, there exists $x\in G$ such that $I_{a_n}d(x,x_n)=0$ for each $n\in \mathbb{N}$ (similarly, such an element $x$ is unique and usually denoted by $\sum_{n}I_{a_n}x_n$). By the way, if $G$ is a $d$-$\sigma$-stable subset of $S$ and $H$ is a nonempty subset of $G$, $\sigma(H):=\{\sum_{n}I_{a_n}h_n: \{a_n,n\in \mathbb{N}\} ~$is a partition of unity in $B_{\mathcal{F}}$ and$~ \{h_n,n\in \mathbb{N}\} ~$is a sequence in$~H \}$ is called the $d$-$\sigma$-stable hull of $H$.
\end{defn}

Theorem 2.7 of \cite{GWYZ} shows that a nonempty subset $G$ of an RM space $(S,d)$ with base $(\Omega,\mathcal{F},\mu)$ is $d$-stable iff, for any $n\in \mathbb{N}$, any $n$-element subset $\{x_k:k=1\sim n\}$ of $G$ and any $n$-partition $\{a_k:k=1\sim n\}$ of unity in $B_{\mathcal{F}}$, there exists an unique $x$ in $G$ such that $I_{a_k}d(x,x_k)=0$ for each $k=1\sim n$. Some important examples of $d$-$\sigma$-stable sets were also provided in \cite{GWYZ}. Besides, it is also obvious that the notions of $d$-$\sigma$-stability ($d$-stability) and $\sigma$-stability (accordingly, stability) coincide for a nonempty subset of an RN module.

\begin{defn}\label{def.2.15}
	Let $(S_1,d_1)$ and $(S_2,d_2)$ be two RM spaces with base $(\Omega,\mathcal{F},\mu)$, $G_1$ and $G_2$ two nonempty subsets of $S_1$ and $S_2$, respectively, and $T:G_1\rightarrow G_2$ a mapping. $T$ is said to be stable if $G_1$ and $G_2$ are $d$-stable and $T(I_{a}x+I_{a^c}y)=I_{a}T(x)+I_{a^c}T(y)$ for any  $x$ and $y$ in $G_1$ and any $a\in B_{\mathcal{F}}$. $T$ is said to be $\sigma$-stable if $G_1$ and $G_2$ are $d$-$\sigma$-stable and $T(\sum_{n}I_{a_n}x_n)=\sum_{n}I_{a_n}T(x_n)$ for any sequence $\{x_n,n\in \mathbb{N}\}$ in $G_1$ and any partition $\{a_n,n\in \mathbb{N}\}$ of unity in $B_{\mathcal{F}}$.
\end{defn}

Lemma \ref{lem.2.16} below is very simple but it considerably simplifies and even makes the proof of the main result--Theorem \ref{thm.2.17} obvious.

\begin{lem}\label{lem.2.16}
	Let $(S_1,d_1)$, $(S_2,d_2)$, $G_1$, $G_2$ and $T$ be the same as in Definition \ref{def.2.15}. Then we have the following statements:
	\begin{itemize}
		\item [(1)] If $G_1$ and $G_2$ are $d$-stable ($d$-$\sigma$-stable) and $T$ is stable ($\sigma$-stable), then $T(G_1)$ is $d$-stable ($d$-$\sigma$-stable).
		\item [(2)]	If $G_1$ and $G_2$ are $d$-stable ($d$-$\sigma$-stable) and $T$ is $L^0$-Lipschitzian (namely, there exists $\xi\in L^0_{+}(\mathcal{F})$ such that $d_2(T(x),T(y))\leq \xi d_1(x,y)$ for any $x$ and $y$ in $G_1$), then $T$ is stable ($\sigma$-stable).
		\item [(3)]	If $G_1$ and $G_2$ are $d$-stable ($d$-$\sigma$-stable) and $T$ is a stable ($\sigma$-stable) bijective mapping from $G_1$ onto $G_2$, then $T^{-1}$ is also stable ($\sigma$-stable).
	\end{itemize}
\end{lem}
\begin{proof}
	These statements can be checked merely by definition, so are omitted.
\end{proof}

\begin{thm}\label{thm.2.17}
	Let $(S,d)$  be an RM space with base $(\Omega,\mathcal{F},\mu)$, $\mathcal{U}_{b}(S,L^0(\mathcal{F}))$ and $T$ the same as in Theorem \ref{thm.2.6}. Then a nonempty subset $G$ of $S$ is $d$-$\sigma$-stable ($d$-stable) iff $T(G)$ is $\sigma$-stable (stable), and hence the notion of $d$-$\sigma$-stability ($d$-stability) can be regarded a special case of that of $\sigma$-stability (stability).
\end{thm}
\begin{proof}
	By (2) of Theorem \ref{thm.2.6}, $T$ is isometric and thus also $L^0$-Lipschitzian. Further, $T(G)$ is $\sigma$-stable (stable) in $(\mathcal{U}_{b}(S,L^0(\mathcal{F})), \| \cdot\|)$ iff $T(G)$ is $d$-$\sigma$-stable ($d$-stable) with respect to the random metric induced by $ \| \cdot\|$, then the proof of the theorem follows immediately from Lemma \ref{lem.2.16}.
\end{proof}

Theorem \ref{thm.2.17} shows that the study of isometric invariants (for example, $\sigma$-stability) in random metric spaces can be converted into the study of the corresponding problems in RN modules. Following are some interesting applications of such an idea.

\begin{thm}\label{thm2.18}
	Let $(S,d)$  be an RM space with base $(\Omega,\mathcal{F},\mu)$ and $G$ be a $d$-stable subset of $S$. Then $G$ is $(\varepsilon,\lambda)$-complete iff $G$ is both $d$-$\sigma$-stable and $L^0$-complete.
\end{thm}
\begin{proof}
	Since $d$-stability, $d$-$\sigma$-stability, $(\varepsilon,\lambda)$-completeness and $L^0$-completeness are all isometric invariants with respect to random metric, we only need to prove that the proposition holds for an RN module. Just as stated in the end of Example \ref{exm.2.11}, the proposition has been proved for a random locally convex module, it, of course holds also for an RN module.
\end{proof}

The classical Hausdorff theorem states that a nonempty subset $G$ of a metric space is sequentially compact iff $G$ is both complete and totally bounded. The theorem has been generalized to an RN module in a nontrivial manner: Theorem 2.3 of \cite{GWCXY} states that, for a $\sigma$-stable subset $G$ of an RN module, $G$ is random sequentially compact iff $G$ is both random totally bounded and  $\mathcal{T}_{\varepsilon,\lambda}$-complete. If we can give the proper notions of random sequentially compactness and random total boundedness in an RM space, then the classical Hausdorff theorem can be also generalized to RM spaces. Definition \ref{def.2.19} below will do the thing!

Let $(S,d)$ be a $d$-$\sigma$-stable RM space with base $(\Omega,\mathcal{F},\mu)$. First, let us introduce the following terminologies to make preparations for Definition \ref{def.2.19} and Theorem \ref{thm.2.21}.
\begin{itemize}
	\item [(1)] For any sequence $\{G_n,n\in \mathbb{N}\}$ of nonempty subsets of $S$ and any partition $\{a_n,n\in \mathbb{N}\}$ of unity in $B_{\mathcal{F}}$, $\sum I_{a_n}G_n$ stands for the set $\{\sum I_{a_n}x_n: x_n\in G_n ~$for each$~n\in \mathbb{N}\}$. A family $\mathcal{E}$ of nonempty subsets of $S$ is said to be $d$-$\sigma$-stable if $\sum I_{a_n}G_n$ still belongs to $\mathcal{E}$ for any sequence $\{G_n,n\in \mathbb{N}\}$ in $\mathcal{E}$ and any partition $\{a_n,n\in \mathbb{N}\}$ of unity in $B_{\mathcal{F}}$. $\{B(x,r):x\in S ~$and$~ r\in L^0_{++}(\mathcal{F})\}$ is a $d$-$\sigma$-stable family since it is easy to check $\sum_{n}I_{a_n}B(x_n,r_n)=B(\sum_{n}I_{a_n}x_n,\sum_{n}I_{a_n}r_n)$ for any sequence $\{x_n,n\in \mathbb{N}\}$ in $S$ and any sequence $\{r_n,n\in \mathbb{N}\}$ in $L^0_{++}(\mathcal{F})$. Here $B(x,r)=\{y\in S:d(x,y)<r ~$on$~\Omega\}$ for any $x$ in $S$ and $r\in L^0_{++}(\mathcal{F})$.
	\item [(2)] Let $L^0(\mathcal{F},\mathbb{N})$ be the set of equivalence classes of measurable functions from $(\Omega,\mathcal{F},\mu)$ to $\mathbb{N}$, which is, obviously, a $\sigma$-stable directed set as a subset of $(L^0(\mathcal{F}),\leq)$. A $\sigma$-stable mapping $x$ from $L^0(\mathcal{F},\mathbb{N})$ to $S$ is called a stable sequence in $S$, denoted by $\{x_n:=x(n),n\in L^0(\mathcal{F},\mathbb{N})\}$. A stable sequence $\{y_m,m\in L^0(\mathcal{F},\mathbb{N})\}$ is called a stable subsequence of a stable sequence $\{x_n,n\in L^0(\mathcal{F},\mathbb{N})\}$ if there exists a $\sigma$-stable mapping $l$ from $L^0(\mathcal{F},\mathbb{N})$ to $L^0(\mathcal{F},\mathbb{N})$ satisfying $y_m=x_{l_m}$ for each $m\in L^0(\mathcal{F},\mathbb{N})$ (where $l_m=l(m)$) at the same time $\{y_m,m\in L^0(\mathcal{F},\mathbb{N})\}$ is also a subnet of $\{x_n,n\in L^0(\mathcal{F},\mathbb{N})\}$. As usual, we also write $\{x_{l_m} ,m\in L^0(\mathcal{F},\mathbb{N})\}$ for $\{y_m,m\in L^0(\mathcal{F},\mathbb{N})\}$.
	\item [(3)] For any sequence $\{x_k,k\in \mathbb{N}\}$ in $S$ and any $n\in L^0(\mathcal{F},\mathbb{N})$, letting $x_{n}=\sum_{k}I_{[n=k]}x_k$, then $\{x_n,n\in L^0(\mathcal{F},\mathbb{N})\}$ is a stable sequence in $S$, called the stable sequence generated by $\{x_k,k\in \mathbb{N}\}$, where $[n=k]$ is the equivalence class of $(n^0=k):=\{\omega \in \Omega:n^0(\omega)=k\}$ for an arbitrarily chosen representative $n^0$ of $n$, $\{[n=k],k\in \mathbb{N}\}$ clearly forms a partition of unity in $B_{\mathcal{F}}$. Conversely, for any given stable sequence $\{x_n,n\in L^0(\mathcal{F},\mathbb{N})\}$, denoting $x_n$ by $x_k$ when $n=\tilde{I}_{\Omega} \cdot k$ for some $k\in \mathbb{N}$ (namely $n$ is the equivalence class of the constant function with value $k$ on $\Omega$), then $\{x_n,n\in L^0(\mathcal{F},\mathbb{N})\}$ is exactly generated by $\{x_k,k\in \mathbb{N}\}$.
	\item [(4)] For a sequence $\{x_n,n\in \mathbb{N}\}$ in $S$, a sequence $\{y_k,k\in \mathbb{N}\}$ in $S$ is called a random subsequence of $\{x_n,n\in \mathbb{N}\}$ if there exist a sequence $\{n_k,k\in \mathbb{N}\}$ in $L^0(\mathcal{F},\mathbb{N})$ such that $y_k=x_{n_k}$ and $n_k<n_{k+1}$ on $\Omega$ for each $k\in \mathbb{N}$. A careful reader will find that we require $n_k$ to be a positive integer-valued measurable function in the definition of a random subsequence in \cite{GWCXY} instead of an element in $L^0(\mathcal{F},\mathbb{N})$  here, but it is easy to check that the two formulations are essentially equivalent!
\end{itemize}

\begin{defn}\label{def.2.19}
	A $d$-$\sigma$-stable RM space $(S,d)$ is said to be:
	\begin{itemize}
		\item [(1)] stably compact if every $d$-$\sigma$-stable family of $d$-$\sigma$-stable $\mathcal{T}_{c}$-closed subsets of $S$ has a nonempty intersection whenever the family has finite intersection property (namely, any finite subfamily of it has a nonempty intersection).
		\item [(2)] stably sequentially compact if every stable sequence in $S$ admits a stable subsequence that is convergent in $\mathcal{T}_{c}$.
		\item [(3)] random sequentially compact if, for every sequence $\{x_n,n\in \mathbb{N}\}$ in $S$, there exists a random subsequence $\{x_{n_k},k\in \mathbb{N}\}$ of $\{x_n,n\in \mathbb{N}\}$ such that $\{x_{n_k},k\in \mathbb{N}\}$ converges in $\mathcal{T}_{\varepsilon,\lambda}$.
		\item [(4)] random totally bounded if, for any given $\varepsilon \in L^0_{++}(\mathcal{F})$, there exists a sequence $\{l_n,n\in \mathbb{N}\}$ in $\mathbb{N}$ and a partition $\{a_n,n\in \mathbb{N}\}$ of unity in $B_{\mathcal{F}}$ together with a finite subset $\{x_{nk}: k=1\sim l_n \}$ in $S$ for each $n\in \mathbb{N}$ such that $S=\sum_{n}I_{a_n}G_n$, where $G_n=\bigcup \{B(x,\varepsilon):x\in \sigma(\{x_{nk}: k=1\sim l_n\})\}$ for each $n\in \mathbb{N}$.
	\end{itemize}
\end{defn}

\begin{rem}\label{rem.2.20}
	The notion of stable compactness can also be defined in terms of $d$-$\sigma$-stable families of $d$-$\sigma$-stable $\mathcal{T}_{c}$-open subsets as in \cite{JZ}, but an advantage of (1) of Definition \ref{def.2.19} can avoid the use of the notion of a stably finite set, and as proved in \cite{GWT}, the notion of stable compactness in the sense of Definition \ref{def.2.19} is indeed equivalent to that in the sense of \cite{JZ}. Besides, (1), (2) and (4) of Definition \ref{def.2.19} generalize the corresponding notions in \cite{JZ} where $S$ is required to be a subset of an $L^0(\mathcal{F},\mathbb{K})$-module.
	
	It should be also pointed out that all the notions in Definition \ref{def.2.19} are still well defined for a  $d$-$\sigma$-stable subset $G$ of an RM space $(S,d)$ since we only need to regard $(G,d)$ as a $d$-$\sigma$-stable RM space.
\end{rem}

(1) to (3) of Theorem \ref{thm.2.21} below was first considered for a special case by Jamneshan and Zapata in \cite{JZ} before the notion of a $d$-$\sigma$-stability was presented. Let $\mathcal{M}$ be an $L^{0}(\mathcal{F},\mathbb{K})$-module and $S$ be a $\sigma$-stable subset of $\mathcal{M}$, a random metric $d:S\times S\rightarrow L_{+}^{0}(\mathcal{F})$ on $S$ is said to be $\sigma$-stable if $d(\sum_{n}I_{a_{n}}x_{n},\sum_{n}I_{a_{n}}y_{n})=\sum_{n}I_{a_{n}}d(x_{n},y_{n})$ for any sequence $\{(x_{n},y_{n}),n\in \mathbb{N}\}$ in $S\times S$ and any partition $\{a_{n},n\in \mathbb{N}\}$ of unity in $B_{\mathcal{F}}$, such an $RM$ space $(S,d)$ is called a $\sigma$-stable random metric space in \cite{JZ}. It is easy to see that a $\sigma$-stable $RM$ space $(S,d)$ must be $d$-$\sigma$-stable, an important example of a $\sigma$-stable $RM$ space is $(S,d)$, where $S$ is a $\sigma$-stable subset of an $RN$ module $(\mathcal{M},\|\cdot\|)$ and $d(x,y)=\|x-y\|$.

\begin{thm}\label{thm.2.21}
	Let $(S,d)$ be a $d$-$\sigma$-stable RM space with base $(\Omega,\mathcal{F},\mu)$. Then the following statements are equivalent:
	\begin{itemize}
		\item [(1)] $(S,d)$ is stably compact.
		\item [(2)] $(S,d)$ is stably sequentially compact.
		\item [(3)]	$(S,d)$ is random totally bounded and $L^0$-complete.
		\item [(4)]	$(S,d)$ is random sequentially compact.	
	\end{itemize}
\end{thm}
\begin{proof}
	In fact, Theorem 5.10 of \cite{JZ} shows that $(1)\Leftrightarrow (2)\Leftrightarrow (S,d)$ is random totally bounded and $\mathcal{T}_{c}$-stably sequentially complete (namely, every stable $\mathcal{T}_{c}$-Cauchy sequence in $(S,d)$ converges in $\mathcal{T}_{c}$), when $S$ is a $\sigma$-stable subset of an RN module. Recently, it is proved in Theorem 5.3 of \cite{GWCXY} that a $\sigma$-stable subset $G$ of an RN module is
	$\mathcal{T}_{c}$-stably sequentially complete  iff $G$ is $\mathcal{T}_{c}$-complete, so that $(1)\Leftrightarrow (2)\Leftrightarrow (3)$ hold for a $\sigma$-stable subset of an RN module. Further, Theorem 2.3 of \cite{GWCXY} shows that a $\sigma$-stable subset $G$ of an RN module is
	random sequentially compact iff $G$ is random totally bounded and $\mathcal{T}_{\varepsilon,\lambda}$-complete, so at this time $(3)\Leftrightarrow (4)$ by noticing that $G$ is $\mathcal{T}_{c}$-complete iff G is $\mathcal{T}_{\varepsilon,\lambda}$-complete. To sum up,
	we have proved the theorem when $(S,d)$ is a $\sigma$-stable subset of an $RN$ module, and hence we have completed the proof again by using  Theorem \ref{thm.2.6} and Theorem \ref{thm.2.17}.
\end{proof}

\begin{rem}\label{rem.2.22}
	Theorem \ref{thm.2.21} not only generalizes but also improves Theorem 5.10 of \cite{JZ} in that $(S,d)$ here can be assumed as a general $d$-$\sigma$-stable $RM$ space in advance and (3) of Theorem 2.21 is stated in terms of $L^{0}$-completeness (namely completeness with respect to the $L^{0}$-uniformity) instead of stably sequential completeness as in \cite{JZ} (it is obvious that stably sequential completeness appears to be weaker than $L^{0}$-completeness, although they are essentially equivalent to each other, one more easily understand the notion of $L^{0}$-completeness from a standard analytic perspective). Besides, such notions as stable compactness, stable sequential compactness, random total boundedness (also called stable total boundedness in \cite{JZ}) and $L^{0}$-completeness involved in (1) to (3) of Theorem \ref{thm.2.21} are purely relative to the $L^{0}$-uniformity, whereas the notion of random sequential compactness involved in (4) of Theorem 2.21 is purely relative to the $(\varepsilon,\lambda)$-uniformity, so results like Theorem \ref{thm.2.21} are more interesting since such results reflect the connection between the two uniformities.
\end{rem}

%###################################################################################################
%###################################################################################################
\section{On $\sigma$-stability and the gluing property}
%###################################################################################################
%###################################################################################################

The aim of this section is to prove that the gluing property of an $L^{p}(\mathcal{F})$-normed $L^{\infty}(\mathcal{F},\mathbb{K})$-module can be derived from the $\sigma$-stability of a $\mathcal{T}_{\varepsilon,\lambda}$-complete $RN$ module (see Corollary \ref{coro.3.7} below), which is based on the fact that an $L^{p}(\mathcal{F})$-normed $L^{\infty}(\mathcal{F},\mathbb{K})$-module can be exactly generated by a $\mathcal{T}_{\varepsilon,\lambda}$-complete $RN$ module, see Theorem \ref{thm.3.6} below for details. As applications, we also simplify some discussions of module dual theory developed in \cite{Gigl} by connecting the former with the theory of random conjugate spaces of $RN$ spaces, and in particular we establish the Hahn-Banach theorem and the separation theorem between a point and a  closed $L^{0}$-convex set in an $L^{p}(\mathcal{F})$-normed $L^{\infty}(\mathcal{F},\mathbb{K})$-module, see Propositions \ref{pro.3.13} and \ref{pro.3.14} below. Besides, in the end of this section we also show that Theorem 3.8 and Proposition 3.10 of \cite{LP2} can be regarded as a special case of Corollary \ref{coro.3.10} in this paper. Since $L^{p}(\mathcal{F})$-normed $L^{\infty}(\mathcal{F},\mathbb{K})$-modules were introduced in \cite{Gigl} under the general framework of $L^{\infty}(\mathcal{F},\mathbb{K})$-modules, let us first recall the relevant terminologies as follows.

$L^{p}(\mathcal{F},\mathbb{K})$ denotes the usual Banach space of equivalence classes of $p$-integrable (for $1\leq p<+\infty$) or essentially bounded (for $p=+\infty$) measurable functions from $(\Omega,\mathcal{F},\mu)$ to $\mathbb{K}$, whose $p$-norm is denoted by $|\cdot|_{p}$. It is well known that $(L^{\infty}(\mathcal{F},\mathbb{K}),|\cdot|_{\infty})$ is a Banach algebra over $\mathbb{K}$. We always assume in this section that a normed module $(\mathcal{M},\normmm{\cdot})$ over the Banach algebra $L^{\infty}(\mathcal{F},\mathbb{K})$ satisfies $\normmm{\xi x}\leq |\xi|_{\infty}\normmm{x}$ for any $\xi \in L^{\infty}(\mathcal{F},\mathbb{K})$ and any $x\in \mathcal{M}$, where we use $\normmm{\cdot}$ instead of $\|\cdot\|$ for the norm on $\mathcal{M}$ only for convenience since the notation $\|\cdot\|$ has been used for a random norm on an $RN$ space or an $L^{0}$-norm on an $RN$ module. On the other hand, such notions as $L^{\infty}(\mathcal{F})$-premodule, $L^{\infty}(\mathcal{F})$-modules and $L^{p}(\mathcal{F})$-normed $L^{\infty}(\mathcal{F},\mathbb{K})$-modules were introduced only for the real Banach algebra $L^{\infty}(\mathcal{F},\mathbb{R})$ (simply written as $L^{\infty}(\mathcal{F})$), Definitions \ref{def.3.1} and \ref{def.3.2} are essentially adopted from \cite[Chapter 1]{Gigl} merely by including the complex case.

\begin{defn}\label{def.3.1}
	A normed module $(\mathcal{M},\normmm{\cdot})$ over $(L^{\infty}(\mathcal{F},\mathbb{K}),|\cdot|_{\infty})$ (briefly, $L^{\infty}(\mathcal{F},\mathbb{K})$) is called an $L^{\infty}(\mathcal{F},\mathbb{K})$-premodule if $\mathcal{M}$ is complete (namely, a Banach space). In addition, an $L^{\infty}(\mathcal{F},\mathbb{K})$-premodule $(\mathcal{M},\normmm{\cdot})$ is called an $L^{\infty}(\mathcal{F},\mathbb{K})$-module if it satisfs the following two properties:
	\begin{itemize}
		\item [(1)] Locality. For each $x\in \mathcal{M}$ and each partition $\{a_{n},n\in \mathbb{N}\}$ of unity in $B_{\mathcal{F}}$ such that $I_{a_{n}}x=0$ for any $n\in \mathbb{N}$, then $x=0$;
		\item [(2)]	Gluing property. For each sequence $\{x_{n},n\in \mathbb{N}\}$ in $\mathcal{M}$ and each  partition $\{a_{n},n\in \mathbb{N}\}$ of unity in $B_{\mathcal{F}}$ such that $\{\sum_{i=1}^{n}I_{a_{i}}x_{i}:n\in \mathbb{N}\}$ is bounded, then there exists $x\in \mathcal{M}$ such that $I_{a_{n}}x=I_{a_{n}}x_{n}$ for each $n\in \mathbb{N}$ and $\normmm{x}\leq \varliminf\limits_{n}\normmm{\sum_{i=1}^{n}I_{a_{i}}x_{i}}$.
	\end{itemize}
\end{defn}

One can easily see that locality and gluing property in Definition \ref{def.3.1} are an equivalent formulation of those in Definition 1.2.1 of \cite{Gigl}, where the equivalent new formulation is given for a convenient connection with our previous convention from random functional analysis, for example, locality under the new formulation is very similar to regularity of an $L^{0}(\mathcal{F},\mathbb{K})$-module, and further it is also easy to observe that gluing property is very similar to $\sigma$-stability of an $RN$ module.

\begin{defn}\label{def.3.2}
	Let $1\leq p\leq +\infty$. An $L^{\infty}(\mathcal{F},\mathbb{K})$-premodule $(\mathcal{M},\normmm{\cdot})$ is called an $L^{p}(\mathcal{F})$-normed $L^{\infty}(\mathcal{F},\mathbb{K})$-premodule if there exists a mapping $\|\cdot\|$ from $\mathcal{M}$ to $L_{+}^{p}(\mathcal{F})$ (where $L_{+}^{p}(\mathcal{F})=\{\xi\in L^{p}(\mathcal{F}):\xi\geq 0\}$) such that the following two conditions are satisfied:
	\begin{itemize}
		\item [(1)] $\normmm{x}=|\|x\||_{p}$ for each $x\in \mathcal{M}$;
		\item [(2)]	$\|\xi x\|=|\xi|\|x\|$ for each $\xi\in L^{\infty}(\mathcal{F},\mathbb{K})$ and each $x\in \mathcal{M}$.
	\end{itemize}
	As usual, $\|\cdot\|$ is called the pointwise $L^{p}(\mathcal{F})$-norm (briefly, the pointwise norm) on $\mathcal{M}$, and $\normmm{\cdot}$ is simply written as $\normmm{\cdot}_{p}$.\\
	In addition, an $L^{p}(\mathcal{F})$-normed $L^{\infty}(\mathcal{F},\mathbb{K})$-premodule is called an $L^{p}(\mathcal{F})$-normed $L^{\infty}(\mathcal{F},\mathbb{K})$-module if it is also $L^{\infty}(\mathcal{F},\mathbb{K})$-module (namely, it also satisfies both locality and gluing property).
\end{defn}

\begin{rem}\label{rem.3.3}
	Let $(\mathcal{M},\normmm{\cdot}_{p})$ be an $L^{p}(\mathcal{F})$-normed $L^{\infty}(\mathcal{F},\mathbb{K})$-premodule (with the pointwise norm $\|\cdot\|$).
	\begin{itemize}
		\item [(1)] Since existence of the pointwise norm $\|\cdot\|$ is assumed, Gigli has proved in \cite[Proposition 1.2.12]{Gigl} that $\mathcal{M}$ always possesses the locality property, and that $\mathcal{M}$ also always has the gluing property when $p<+\infty$.
		\item [(2)]	While $ii)$ of \cite[Proposition 1.2.12]{Gigl} also has shown that the pointwise norm $\|\cdot\|$ satisfies (RN-3): $\|x+y\|\leq \|x\|+\|y\|$ for any $x$ and $y$ in $\mathcal{M}$, so $(\mathcal{M},\|\cdot\|)$ automatically becomes an $RN$ space over $\mathbb{K}$ with base $(\Omega,\mathcal{F},\mu)$ with the two additional properties:
		\begin{itemize}
			\item [(2.1)] $\mathcal{M}$ is a left module over the algebra $L^{\infty}(\mathcal{F},\mathbb{K})$;
			\item [(2.2)] $\|\xi x\|=|\xi|\|x\|$ for any $\xi\in L^{\infty}(\mathcal{F},\mathbb{K})$ and any $x\in \mathcal{M}$.
		\end{itemize}
		\item [(3)]	Since $(\mathcal{M},\|\cdot\|)$ has the nice properties stated above in (2), letting $\{x_{n},n\in \mathbb{N}\}$ be any given sequence in $\mathcal{M}$ and $\{a_{n},n\in \mathbb{N}\}$ be any given partition of unity in $B_{\mathcal{F}}$ such that $\{\sum_{i=1}^{n}I_{a_{i}}x_{i},n\in \mathbb{N}\}$ is bounded, then it is easy to observe that the following statements hold as Gigli has shown in $iv)$ of \cite[Proposition 1.2.12]{Gigl}:
		\begin{itemize}
			\item [(3.1)] $\{\sum_{i=1}^{n}I_{a_{i}}\|x_{i}\|:=\|\sum_{i=1}^{n}I_{a_{i}}x_{i}\|,n\in \mathbb{N}\}$ is an a.e. nondecreasing sequence in $L_{+}^{0}(\mathcal{F})$, which means that $\lim\limits_{n\rightarrow \infty}\normmm{\sum_{i=1}^{n}I_{a_{i}}x_{i}}_{p}$ always exists and is finite.
			\item [(3.2)] when $p<+\infty$, $\{\sum_{i=1}^{n}I_{a_{i}}x_{i},n\in \mathbb{N}\}$ is a $\normmm{\cdot}_{p}$-Cauchy sequence, and hence convergent to some $x$ in $\mathcal{M}$, satisfying $I_{a_{n}}x_{n}=I_{a_{n}}x$ for any $n\in \mathbb{N}$ and $\normmm{x}_{p}=\lim\limits_{n\rightarrow \infty}\normmm{\sum_{i=1}^{n}I_{a_{i}}x_{i}}_{p}$.
			\item [(3.3)] When $p=+\infty$, if there exists $x\in \mathcal{M}$ such that $I_{a_{n}}x_{n}=I_{a_{n}}x$ for each $n\in \mathbb{N}$, it also always holds that $\normmm{x}_{\infty}=\lim\limits_{n\rightarrow \infty}\normmm{\sum_{i=1}^{n}I_{a_{i}}x_{i}}_{\infty}$ although $x$ is not necessarily the $\normmm{\cdot}_{\infty}$-limit of $\{\sum_{i=1}^{n}I_{a_{i}}x_{i},n\in \mathbb{N}\}$.	
		\end{itemize}
		\item [(4)] By the above-stated analysis, the gluing property of an $L^{p}$-normed $L^{\infty}(\mathcal{F},\mathbb{K})$-module $(\mathcal{M},\normmm{\cdot}_{p})$ can be equivalently formulated as: for each sequence $\{x_{n},n\in \mathbb{N}\}$ in $\mathcal{M}$ and each partition $\{a_{n},n\in \mathbb{N}\}$ of unity in $B_{\mathcal{F}}$ such that $\{\sum_{i=1}^{n}I_{a_{i}}x_{i},n\in \mathbb{N}\}$ is bounded, then there exists $x\in \mathcal{M}$ such that $I_{a_{n}}x=I_{a_{n}}x_{n}$ for each $n\in \mathbb{N}$, since at this time it always automatically holds that $\normmm{x}_{p}\leq \varliminf\limits_{n} \normmm{\sum_{i=1}^{n}I_{a_{i}}x_{i}}_{p}$.	
	\end{itemize}
 \end{rem}

\begin{ex}\label{ex.3.4}
	Let $1\leq p\leq +\infty$ and $(S,\|\cdot\|)$ be a an $RN$ module over $\mathbb{K}$ with base $(\Omega,\mathcal{F},\mu)$, further let $L^{p}(S)=\{x\in S:\normmm{x}_{p}<+\infty\}$, then $(L^{p}(S),\normmm{\cdot}_{p})$ is a normed module over $(L^{\infty}(\mathcal{F},\mathbb{K}),|\cdot|_{\infty})$, called the one generated by $(S,\|\cdot\|)$, where $\normmm{x}_{p}=|\|x\||_{p}$ for any $x\in L^{p}(S)$; furthermore, it is also easy to see that $L^{p}(S)$ is $\mathcal{T}_{\varepsilon,\lambda}$-dense in $(S,\|\cdot\|)$. When $(S,\|\cdot\|)$ is $\mathcal{T}_{\varepsilon,\lambda}$-complete, it is very easy to check that $(L^{p}(S),\normmm{\cdot}_{p})$ is an $L^{p}(\mathcal{F})$-normed $L^{\infty}(\mathcal{F},\mathbb{K})$-module, whose gluing property is checked as follows: for each sequence $\{x_{n},n\in \mathbb{N}\}$ and each partition $\{a_{n},n\in \mathbb{N}\}$ of unity in $B_{\mathcal{F}}$ such that $\sup\{\normmm{\sum_{i=1}^{n}I_{a_{i}}x_{i}}_{p},n\in \mathbb{N}\}<+\infty$, it follows immediately from the $\sigma$-stability of $S$ that there exists $x\in S$ such that $I_{a_{n}}x=I_{a_{n}}x_{n}$ for each $n\in \mathbb{N}$, since it always holds that $\normmm{x}_{p}=\lim\limits_{n\rightarrow \infty}\normmm{\sum_{i=1}^{n}I_{a_{i}}x_{i}}_{p}<+\infty$, $x$ must belong to $L^{p}(S)$. In particular, when $p=+\infty,~\normmm{x}_{\infty}=\sup\{\normmm{\sum_{i=1}^{n}I_{a_{i}}x_{i}}_{\infty},n\in \mathbb{N}\}=\sup\{\normmm{I_{a_{n}}x_{n}}_{\infty},n\in \mathbb{N}\}$.
\end{ex}

It is well known that a $\sigma$-finite measure is always equivalent to a probability measure, thus when the problems we consider only involve the equivalence of measures, we usually assume the base space of an $RN$ module to be a probability space to emphasize the randomness of an $RN$ module. We also earlier realized that this restriction on the base space is not sufficiently enough for applications of $RN$ modules in analysis, for example, when we are forced to consider some problems involving integrals with respect to a measure, so we also developed the theory of $RN$ modules with an arbitrary or $\sigma$-finite measure space as base in \cite{Guo2,Guo5}, and in particular established the connection between the random conjugate space $S^{*}$ of an $RN$ module $(S,\|\cdot\|)$ and the conjugate space $L^{p}(S)'$ of the normed space $L^{p}(S)$ generated by $(S,\|\cdot\|)$ when $1\leq p<+\infty$, see \cite[Theorem 3.1]{Guo5} for details. In particular, we have Proposition \ref{pro.3.5} below as a special case of Theorem 3.1 of \cite{Guo5} (namely, when $(\Omega,\mathcal{F},\mu)$ is $\sigma$-finite).

\begin{prop}\label{pro.3.5}
	Let $(S,\|\cdot\|)$ be an $RN$ module over $\mathbb{K}$ with base $(\Omega,\mathcal{F},\mu)$ and $1\leq p<+\infty$, further let $(L^{p}(S),\normmm{\cdot}_{p})$ the normed space over $\mathbb{K}$ consisting of all $x\in S$ such that $\normmm{x}_{p}:=|\|x\||_{p}<+\infty$, then $(L^{q}(S^{*}),\normmm{\cdot}_{q})$ is isometrically isomorphic onto $(L^{p}(S),\normmm{\cdot}_{p})'$ under the canonical mapping $T:L^{q}(S^{*})\rightarrow L^{p}(S)'$ (denote $T(f)$ by $T_{f}$ for any $f\in L^{q}(S^{*})$) defined by $T_{f}(x)=\int_{\Omega}f(x)d\mu$ for any $x\in L^{p}(S)$, where $\frac{1}{p}+\frac{1}{q}=1$, namely $q$ is the H\"{o}lder conjugate number of $p$, and $(L^{q}(S^{*}),\normmm{\cdot}_{q})$ is the normed space generated by $(S^{*},\|\cdot\|)$ (in fact, it is also a Banach space since $(S^{*},\|\cdot\|)$ is always $\mathcal{T}_{\varepsilon,\lambda}$-complete).
\end{prop}

Theorem \ref{thm.3.6} below shows that every $L^{p}(\mathcal{F})$-normed $L^{\infty}(\mathcal{F},\mathbb{K})$-module is exactly generated by a $\mathcal{T}_{\varepsilon,\lambda}$-complete $RN$ module $(S,\|\cdot\|)$ over $\mathbb{K}$ with base $(\Omega,\mathcal{F},\mu)$, that is to say, every $L^{p}(\mathcal{F})$-normed $L^{\infty}(\mathcal{F},\mathbb{K})$-module can be written as $(L^{p}(S),\normmm{\cdot}_{p})$ for some unique $\mathcal{T}_{\varepsilon,\lambda}$-complete $RN$ module $(S,\|\cdot\|)$ over $\mathbb{K}$ with base $(\Omega,\mathcal{F},\mu)$. Theorem \ref{thm.3.6} is essentially due to Gigli since he already pointed out this result in \cite[p.31]{Gigl} without a proof. Here, we give the details of the proof of Theorem \ref{thm.3.6} by using Proposition \ref{pro.3.5} and Egoroff's theorem since Theorem \ref{thm.3.6} will pave the way for this whole section!

\begin{thm}\label{thm.3.6}
	Let $(\mathcal{M},\normmm{\cdot}_{p})$ be an $L^{p}(\mathcal{F})$-normed $L^{\infty}(\mathcal{F},\mathbb{K})$-module for any given $p\in [1,+\infty]$, then there exists a $\mathcal{T}_{\varepsilon,\lambda}$-complete $RN$ module $(S,\|\cdot\|)$ over $\mathbb{K}$ with base $(\Omega,\mathcal{F},\mu)$ such that the following properties are satisfied:
	\begin{itemize}
		\item [(1)] $(\mathcal{M},\normmm{\cdot}_{p})=(L^{p}(S),\normmm{\cdot}_{p})$.
		\item [(2)] $(S,\|\cdot\|)$ is unique in the sense of isometric isomorphism of $RN$ modules.
		\item [(3)] The $L^{0}$-norm $\|\cdot\|$ on $S$ is an extension of the pointwise norm on $\mathcal{M}$ and the module multiplication on the $L^{0}(\mathcal{F},\mathbb{K})$-module $S$ is an extension of the module multiplication on the $L^{\infty}(\mathcal{F},\mathbb{K})$-module $\mathcal{M}$.
		\item [(4)]	$\mathcal{M}$ is $\mathcal{T}_{\varepsilon,\lambda}$-dense in $S$.		
	\end{itemize}
\end{thm}
\begin{proof}
	Let $\|\cdot\|$ be the pointwise norm on $\mathcal{M}$. Since $(\mathcal{M},\|\cdot\|)$ is an $RN$ space over $\mathbb{K}$ with base $(\Omega,\mathcal{F},\mu)$ such that $\mathcal{M}$ is also a left module  over $L^{\infty}(\mathcal{F},\mathbb{K})$ and $\|\xi x\|=|\xi|\|x\|$ for any $\xi\in L^{\infty}(\mathcal{F},\mathbb{K})$ and any $x\in \mathcal{M}$. By Proposition \ref{pro.2.9} $(\mathcal{M},\|\cdot\|)$ has a $\mathcal{T}_{\varepsilon,\lambda}$-completion $(S,\|\cdot\|^{\tilde{}})$, Gigli \cite[p.31]{Gigl} already showed that $(S,\|\cdot\|^{\tilde{}})$ becomes an $\mathcal{T}_{\varepsilon,\lambda}$-complete $RN$ module over $\mathbb{K}$ with base $(\Omega,\mathcal{F},\mu)$ in a natural way, which clearly satisfies (2), (3) and (4) if we still denote $\|\cdot\|^{\tilde{}}$ by $\|\cdot\|$. It remains to prove (1). It is also obvious that $(\mathcal{M},\normmm{\cdot}_{p})$ is a complete subspace of $(L^{p}(S),\normmm{\cdot}_{p})$. We will prove (1) by dividing two cases according to $p<+\infty$ or $p=+\infty$ as follows.
	
	When $p<+\infty$, we only need to prove that $\mathcal{M}$ is $\normmm{\cdot}_{p}$-dense in $L^{p}(S)$. For any $F\in L^{p}(S)'=L^{q}(S^{*})$ such that $F(x)=0$ for each $x\in \mathcal{M}$, where $q\in [1,+\infty]$ is such that $\frac{1}{p}+\frac{1}{q}=1$, then by Proposition \ref{pro.3.5} there exists $f\in L^{q}(S^{*})$ such that $F=T_{f}$ so that $\int_{\Omega}f(x)d\mu=0$ for each $x\in \mathcal{M}$. Further, since $\mathcal{M}$ is a left module  over $L^{\infty}(\mathcal{F},\mathbb{K})$, $\tilde{I}_{A}x\in \mathcal{M}$ for each $x\in \mathcal{M}$ and each $A\in \mathcal{F}$, then $\int_{A}f(x)d\mu=\int_{\Omega}f(\tilde{I}_{A}x)d\mu=0$ for each $x\in \mathcal{M}$ and each $A\in \mathcal{F}$, which implies $f(x)=0$ for each $x\in \mathcal{M}$, so that $f(x)=0$ for each $x\in S$ since $\mathcal{M}$ is $\mathcal{T}_{\varepsilon,\lambda}$-dense in S and $f\in S^{*}$, namely $f=0$, in other words, $\mathcal{M}$ is $\normmm{\cdot}_{p}$-dense in $L^{p}(S)$.
	
	When $p=+\infty$, the case is slightly more complicated since we are forced to adopt a different method. For any given $y\in L^{\infty}(S)$, there exists a sequence $\{x_{n},n\in \mathbb{N}\}$ in $\mathcal{M}$ such that $\{x_{n},n\in \mathbb{N}\}$ is $\mathcal{T}_{\varepsilon,\lambda}$-convergent to $y$ since $\mathcal{M}$ is $\mathcal{T}_{\varepsilon,\lambda}$-dense in $S$ and hence also $\mathcal{T}_{\varepsilon,\lambda}$-dense in $L^{\infty}(S)$. According to the definition of $\mathcal{T}_{\varepsilon,\lambda}$, $\{\|x_{n}-y\|,n\in \mathbb{N}\}$ converges in probability measure $P_{\mu}$ to $0$, then there exists a subsequence of $\{x_{n},n\in \mathbb{N}\}$ (we still denote the subsequence by $\{x_{n},n\in \mathbb{N}\}$) such that $\{\|x_{n}-y\|,n\in \mathbb{N}\}$ is convergent almost uniformly to $0$ by Egoroff's theorem, namely, for each $k\in \mathbb{N}$, there exists $E_{k}\in \mathcal{F}$ such that $P_{\mu}(E_{k})>1-\frac{1}{k}$ and $\{\|x_{n}-y\|,n\in \mathbb{N}\}$ uniformly converges to $0$ on $E_{k}$ (namely, $\{\|\tilde{I}_{E_{k}}x_{n}-\tilde{I}_{E_{k}}y\|,n\in \mathbb{N}\}$ uniformly converges to $0$ on $\Omega$), which implies that $\{\normmm{\tilde{I}_{E_{k}}x_{n}-\tilde{I}_{E_{k}}y}_{\infty},n\in \mathbb{N}\}$ converges to $0$ for each $k\in \mathbb{N}$. Letting $E_{0}=\emptyset$ and $A_{k}=E_{k}\verb|\| (\cup_{l=1}^{k-1}E_{k})$ for each $k\in \mathbb{N}$, then $\{A_{k},k\in \mathbb{N}\}$ forms a countable partition of $\Omega$ to $\mathcal{F}$, namely, $\{a_{k},k\in \mathbb{N}\}$ is a partition of unity in $B_{\mathcal{F}}$, where $a_{k}=[A_{k}]$ (the equivalence class of $A_{k}$) for each $k\in \mathbb{N}$, it also, of course, holds that $\{\normmm{I_{a_{k}}x_{n}-I_{a_{k}}y}_{\infty},n\in \mathbb{N}\}$ converges to $0$ for each $k\in \mathbb{N}$. Since $\{I_{a_{k}}x_{n},n\in \mathbb{N}\}$ is a sequence in $\mathcal{M}$ for each given $k\in \mathbb{N}$ and $(\mathcal{M},\normmm{\cdot}_{\infty})$ is complete, then $I_{a_{k}}y\in \mathcal{M}$ for each $k\in \mathbb{N}$. Further, letting $y_{k}=I_{a_{k}}y$ for any $k\in \mathbb{N}$, then it is clear that $\sup\{\normmm{y_{k}}_{\infty}:k\in \mathbb{N}\}=\normmm{y}_{\infty}<+\infty$, then by the gluing property of $\mathcal{M}$ there exists $x\in \mathcal{M}$ such that $I_{a_{k}}y_{k}=I_{a_{k}}x$, namely $I_{a_{k}}y=I_{a_{k}}x$ for each $k\in \mathbb{N}$, which, of course, implies that $\|x-y\|=(\sum_{k=1}^{\infty}I_{a_{k}})\|x-y\|=\sum_{k=1}^{\infty}I_{a_{k}}\|x-y\|=\sum_{k=1}^{\infty}\|I_{a_{k}}x-I_{a_{k}}y\|=0$, that is to say, $y=x\in \mathcal{M}$.
\end{proof}

\begin{cor}\label{coro.3.7}
	The gluing property of an $L^{p}(\mathcal{F})$-normed $L^{\infty}(\mathcal{F},\mathbb{K})$-module $(\mathcal{M},\normmm{\cdot}_{p})$ can be derived from the $\sigma$-stability of the generating  $\mathcal{T}_{\varepsilon,\lambda}$-complete $RN$ module $(S,\|\cdot\|)$ which generates $\mathcal{M}$ (namely, $\mathcal{M}=L^{p}(S)$).
\end{cor}
\begin{proof}
	In Example \ref{ex.3.4}, we have proved that the gluing property of $L^{p}(S)$ can be derived from $\sigma$-stability of $S$, so the gluing property of $\mathcal{M}$ can be derived from $\sigma$-stability of $S$.
\end{proof}

Definition \ref{def.3.8} below of continuous module homomorphisms and in particular of wise module duals was introduced for $L^{\infty}(\mathcal{F},\mathbb{K})$-modules in \cite[Chapter 1]{Gigl}, in fact it is still well defined for any normed modules over the Banach algebra $L^{\infty}(\mathcal{F},\mathbb{K})$, where we adopt such a general definition to provide a convenience for stating the Hahn-Banach theorem related to a module dual of an $L^{p}(\mathcal{F})$-normed $L^{\infty}(\mathcal{F},\mathbb{K})$-module, see Propositions \ref{pro.3.13} and \ref{pro.3.14} below.

\begin{defn}\label{def.3.8}
	Let $(\mathcal{M},\normmm{\cdot})$ and $(\mathcal{M}_{1},\normmm{\cdot}^{1})$ be two normed modules over the Banach algebra $L^{\infty}(\mathcal{F},\mathbb{K})$. A mapping $T:\mathcal{M}\rightarrow \mathcal{M}_{1}$ is called a continuous module homomorphism if $T$ is a continuous linear operator such that $T(\xi x)=\xi(T(x))$ for any $\xi\in L^{\infty}(\mathcal{F},\mathbb{K}) $ and any $x\in \mathcal{M}$. In addition, $T$ is called an isometric isomorphism if $T$ is bijective and isometric such that both $T$ and $T^{-1}$ are a continuous module homomorphisms, at this time $\mathcal{M}$ and $\mathcal{M}_{1}$ are said to be isometrically isomorphic. Denote by $Hom(\mathcal{M},\mathcal{M}_{1})$ the set of continuous module homomorphisms from $\mathcal{M}$ to $\mathcal{M}_{1}$, it is easy to check that $Hom(\mathcal{M},\mathcal{M}_{1})$ becomes a normed module over $L^{\infty}(\mathcal{F},\mathbb{K})$ when it is endowed with the operator norm and the module multiplication $(\xi T)(x)=\xi(T(x))$ for any $\xi\in L^{\infty}(\mathcal{F},\mathbb{K})$, any $T\in Hom(\mathcal{M},\mathcal{M}_{1})$ and any $x\in \mathcal{M}$. $Hom(\mathcal{M},L^{1}(\mathcal{F},\mathbb{K}))$ is called the module dual of $\mathcal{M}$, denoted by $\mathcal{M}_{m}^{*}$. Here we employ the notation $\mathcal{M}_{m}^{*}$ instead of $\mathcal{M}^{*}$ as in \cite{Gigl} mainly because $\mathcal{M}$ sometimes happens to be also an $RN$ space and previously we have used $\mathcal{M}^{*}$ for the random conjugate space of $\mathcal{M}$.
	
\end{defn}

When $\mathcal{M}$ and $\mathcal{N}$ are both $L^{\infty}(\mathcal{F},\mathbb{K})$-modules, Gigli already proved that $Hom(\mathcal{M},\mathcal{N})$ is also an $L^{\infty}(\mathcal{F},\mathbb{K})$-module, see \cite[pp.16-17]{Gigl}. Proposition \ref{pro.3.9} below is merely a restatement of $v)$ of Proposition 1.2.12, which shows that when $\mathcal{M}$ and $\mathcal{N}$ are $L^{p_{1}}(\mathcal{F})$-normed and $L^{p_{2}}(\mathcal{F})$-normed $L^{\infty}(\mathcal{F},\mathbb{K})$-modules, respectively, where $p_{1}\geq p_{2}\in [1,+\infty]$, $Hom(\mathcal{M},\mathcal{N})$ is an $L^q(\mathcal{F})$-normed $L^{\infty}(\mathcal{F},\mathbb{K})$-module. For this, let us recall some basic facts on continuous module homomorphisms between RN modules. Let $(S,\| \cdot\|)$ and $(S_1,\| \cdot\|)$ be two RN modules over $\mathbb{K}$ with base $(\Omega,\mathcal{F},\mu)$, further let $B(S,S_{1})$ be the $L^{0}(\mathcal{F},\mathbb{K})$-module of continuous module homomorphisms from $(S,\mathcal{T}_{\varepsilon,\lambda})$ to $(S _{1},\mathcal{T}_{\varepsilon,\lambda})$, then it is known in \cite{Guo6} that $T\in B(S,S_{1})$ iff $T$ is an a.e bounded linear operator from $S$ to $S_{1}$, namely $T$ is linear and there exists $\xi\in L^0_{+}(\mathcal{F})$ such that $\|T(x)\|\leq \xi \|x\|$ for any $x\in S$, and at this time $\|T\|:=\bigwedge\{\xi\in L^0_{+}(\mathcal{F}): \|T(x)\|\leq \xi \|x\| ~$for any$~ x\in S \}$ is equal to $\bigvee\{\|T(x)\|:x\in S ~$and$~\|x\|\leq 1  \}$, called the $L^0$-norm of $T$. It is also known from \cite{Guo6} that $(B(S,S_{1}),\| \cdot\|)$ is still an RN module  over $\mathbb{K}$ with base $(\Omega,\mathcal{F},\mu)$, and moreover, $B(S,S_{1})$ is $\mathcal{T}_{\varepsilon,\lambda}$-complete when $S_{1}$ is $\mathcal{T}_{\varepsilon,\lambda}$-complete. It is easy to observe that $B(S,L^{0}(\mathcal{F},\mathbb{K}))$ is just $S^*$, namely the random conjugate space of $S$.

Let $S,S_{1}$ and $B(S,S_{1})$ be as stated above. Since $L^{p}(S)$ is $\mathcal{T}_{\varepsilon,\lambda}$-dense in $S$ for any given $p\in [1,+\infty]$, it is also easy to see that $\{x\in L^{p}(S):\|x\|\leq 1\}$ is $\mathcal{T}_{\varepsilon,\lambda}$-dense in $\{x\in S:\|x\|\leq 1\}$, and thus it always holds that $\|T\|=\bigvee\{\|T(x)\|:x\in L^{p}(S)~\text{and}~\|x\|\leq 1\}$ for any $T\in B(S,S_{1})$, it follows immediately from this observation that Proposition \ref{pro.3.9} and Corollary \ref{coro.3.10} below are indeed a restatement of $v)$ of Proposition 1.2.12 and $i)$ of Proposition 1.2.14 of \cite{Gigl}, respectively.

\begin{prop}\label{pro.3.9}
	Let  $p_1\geq p_2\in [1,+\infty]$, $(S_1,\| \cdot\|)$ and $(S_2,\| \cdot\|)$ be two $\mathcal{T}_{\varepsilon,\lambda}$-complete RN modules over $\mathbb{K}$ with base $(\Omega,\mathcal{F},\mu)$, and $q\in[1,+\infty]$ such that $\frac{1}{p_2}=\frac{1}{p_1}+\frac{1}{q}$. Then $(L^{q}(B(S_1,S_2)), \normmm{\cdot}_{q})$ is isometrically isomorphic onto $Hom(L^{p_{1}}(S_1),L^{p_{2}}(S_2))$ under the canonical mapping $Q:L^{q}(B(S_1,S_2))\rightarrow  Hom(L^{p_{1}}(S_1),L^{p_{2}}(S_2))$ defined by $Q(T)=T\mid_{L^{p_{1}}(S_1)}$ for any $T\in L^{q}(B(S_1,S_2))$, where $T\mid_{L^{p_{1}}(S_1)}$ stands for the restriction of $T$ to $L^{p_{1}}(S_1)$ and $L^{q}(B(S_1,S_2))$ is the $L^{q}(\mathcal{F})$-normed $L^{\infty}(\mathcal{F},\mathbb{K})$-module generated by the  $\mathcal{T}_{\varepsilon,\lambda}$-complete RN module $B(S_1,S_2)$.
\end{prop}

Corollary \ref{coro.3.10} below is a restatement of a special case of Proposition \ref{pro.3.9} when $p_1=p$
, $q_2=1$ and $S_2=L^{0}(\mathcal{F},\mathbb{K})$, namely $i)$ of Proposition 1.2.14 of \cite{Gigl}, where we also give a simpler proof of this corollary by using Proposition \ref{pro.3.5} and the idea of proof will be further used in the proof of Proposition \ref{pro.3.13} below.

\begin{cor}\label{coro.3.10}
Let $p\in [1,+\infty]$ and $(S,\|\cdot\|)$ be a $\mathcal{T}_{\varepsilon,\lambda}$-complete $RN$ module over $\mathbb{K}$ with base $(\Omega,\mathcal{F},\mu)$. Then $(L^{q}(S^{*}),\normmm{\cdot}_{q})$ is isometrically isomorphic onto $((L^{p}(S))_{m}^{*},\normmm{\cdot})$ in the sense of $L^{\infty}(\mathcal{F},\mathbb{K})$-modules under the canonical mapping $\hat{T}$ defined by $\hat{T}_{f}(x)=f(x)$ for any $f\in L^{q}(S^{*})$ and any $x\in L^{p}(S)$, where $\hat{T}_{f}$ denotes $\hat{T}(f)$ and $q$ is the H\"{o}lder conjugate number of $p$.
\end{cor}
\begin{proof}
	When $p<+\infty$, for any $T\in (L^{p}(S))_{m}^{*}$, define $Int(T):L^{p}(S)\rightarrow \mathbb{K}$ by $Int(T)(x)=\int_{\Omega}T(x)d\mu$ for any $x\in L^{p}(S)$, then $Int(T)\in L^{p}(S)'$, and hence by Proposition \ref{pro.3.5} there exists $f\in L^{q}(S^{*})$ such that $\int_{\Omega}T(x)d\mu=\int_{\Omega}f(x)d\mu$ for each $x\in L^{p}(S)$ and $\normmm{Int(T)}=\normmm{f}_{q}$ (where $\normmm{Int(T)}$ denotes the norm of $Int(T)$). Further, since $L^{p}(S)$ is an $L^{\infty}(\mathcal{F},\mathbb{K})$-module, $\tilde{I}_{A}x\in L^{p}(S)$ for any $A\in \mathcal{F}$ and $x\in L^{p}(S)$, then $\int_{A}T(x)d\mu=\int_{\Omega}T(\tilde{I}_{A}x)d\mu=\int_{\Omega}f(\tilde{I}_{A}x)d\mu=\int_{A}f(x)d\mu$, so that $T(x)=f(x)$ for any $x\in L^{p}(S)$, namely $T=\hat{T}_{f}$. Furthermore, $\normmm{T}=\normmm{Int(T)}=\normmm{f}_{q}$ by the definition of $\normmm{T}$.
	
	When $p=+\infty$, for any $T\in (L^{\infty}(S))_{m}^{*}$, let $\xi=\bigvee\{|T(x)|:x\in L^{\infty}(S)$ and $\normmm{x}_{\infty}\leq 1\}$, then it is obvious that $\xi=\bigvee\{|T(x)|:x\in S~\text{and}~\|x\|\leq 1\}$. Since it is easy to check $\{|T(x)|:x\in S~\text{and}~\|x\|\leq 1\}$ is directed upwards in $(\bar{L}^{0}(\mathcal{F}),\leq)$, then by Proposition \ref{pro.1.1} there exists a sequence $\{x_{n},n\in \mathbb{N}\}$ in $\{x\in S:\|x\|\leq 1\}$ such that $\{|T(x_{n})|,n\in \mathbb{N}\}$ converges a.e. to $\xi$ in a nondecreasing manner. On the other hand, since $\sup\{\int_{\Omega}|T(x)|d\mu:x\in S~\text{and}~\|x\|\leq 1\}=\sup\{\int_{\Omega}|T(x)|d\mu:x\in L^{\infty}(S)~\text{and}~\normmm{x}_{\infty}\leq 1\}<+\infty$, then $\xi\in L^{1}(\mathcal{F})$ and $|\xi|_{1}=\normmm{T}$. Further, it is easy to observe that $|T(x)|\leq \xi\|x\|$ for any $x\in L^{\infty}(S)$, so $T$ is clearly continuous from $(L^{\infty}(S),\mathcal{T}_{\varepsilon,\lambda})$ to $(L^{1}(\mathcal{F},\mathbb{K}),\mathcal{T}_{\varepsilon,\lambda})$, we can obtain a unique continuous linear extension $f:S\rightarrow L^{0}(\mathcal{F},\mathbb{K})$ such that $|f(x)|\leq \xi\|x\|$ for any $x\in S$. It is obvious that $f\in L^{1}(S^{*})$ and $\|f\|=\xi$, and we also have that $T(x)=f(x)=\hat{T}_{f}(x)$ for any $x\in L^{\infty}(S)$.
\end{proof}

	Let $(\mathcal{M},\normmm{\cdot})$ be a normed module over $L^{\infty}(\mathcal{F},\mathbb{K})$, define $Int: \mathcal{M}_{m}^{*}\rightarrow \mathcal{M}'$ (denotes the conjugate space of $\mathcal{M}$ as a normed space) by $Int(F)(x)=\int_{\Omega}F(x)d\mu$ for an $x\in \mathcal{M}$, where $F$ is any element of $\mathcal{M}_{m}^{*}$, as proved by Gigli in \cite{Gigl}, $Int$ is an isometric embedding. Furthermore, $\mathcal{M}$ is said to have full dual if $Int$ is surjective. Now, define $J_{\mathcal{M}}: \mathcal{M}\rightarrow \mathcal{M}_{m}^{**}:=(\mathcal{M}_{m}^*)_{m}^*$ by $J_{\mathcal{M}}(f)(T)=T(f)$ for any $T\in \mathcal{M}_{m}^{*}$, where $f$ is any element of $\mathcal{M}$. It is also known from \cite{Gigl} that $J_{\mathcal{M}}$ is an isometric embedding when $\mathcal{M}$ has full dual. For any $L^{\infty}(\mathcal{F},\mathbb{K})$-module $\mathcal{M}$ with full dual, $\mathcal{M}$ is said to be module reflexive if $J_{\mathcal{M}}$ is surjective. Proposition 1.2.17 of \cite{Gigl} states that $\mathcal{M}$ is module reflexive  if $\mathcal{M}$ is an $L^{\infty}(\mathcal{F},\mathbb{K})$-module with full dual and $\mathcal{M}$ is reflexive (as a Banach space). Corollary 3.11 is a restatement of Corollary 1.2.18 of \cite{Gigl}, where we restate it by Theorem \ref{thm.3.6} to lay bare the fact that module reflexivity and reflexivity of $L^p(S)$ for $1<p<+\infty$ are both characterized by the random reflexivity of $S$, see e.g. \cite{Guo3, GL} for some deep discussions of random reflexivity.

\begin{cor}\label{coro.3.11}
	Let $1<p<+\infty$ and $(S,\|\cdot\|)$ be a $\mathcal{T}_{\varepsilon,\lambda}$-complete RN module over $\mathbb{K}$ with base $(\Omega,\mathcal{F},\mu)$. Then $( L^p(S), \normmm{\cdot}_{p})$ is module reflexive iff it is reflexive  iff $(S,\|\cdot\|)$ is random reflexive.
\end{cor}
\begin{proof}
		The fact that $( L^p(S), \normmm{\cdot}_{p})$ is module reflexive iff $( L^p(S), \normmm{\cdot}_{p})$ is reflexive is just Corollary 1.2.18 of \cite{Gigl}, while the fact that $( L^p(S), \normmm{\cdot}_{p})$ is reflexive iff $(S,\|\cdot\|)$ is random reflexive is also known, see e.g. \cite{Guo3} or \cite{GL}.
\end{proof}

Although $L^{\infty}(S)$ dose not have full dual in general, Corollary \ref{coro.3.10} shows that $(L^{\infty}(S))_{m}^*=L^{1}(S^*)$, whereas it is well known that $L^{1}(S^*)$ is merely a proper subspace of $L^{\infty}(S)'$ in general, which indeed shows the wisdom of introducing the notion of a module dual. Just as pointed out by Gigl in \cite{Gigl}, the framework of an $L^{\infty}(\mathcal{F},\mathbb{K})$-module is perhaps too general to guarantee that the Hahn-Banach theorem exists for module duals, but the following three propositions show that the algebraic and geometric forms of the Hahn-Banach theorem for module duals of $L^{p}(\mathcal{F})$-normed $L^{\infty}(\mathcal{F},\mathbb{K})$-modules can be obtained by the theory of random conjugate spaces for RN spaces or RN modules. We first restate Corollary 1.2.16 of \cite{Gigl} to Proposition \ref{pro.3.12} and also give a natural and simpler proof by a direct application of the Hahn-Banach theorem for random linear functionals.

\begin{prop}\label{pro.3.12}
	Let $1\leq p<+\infty$ and $(S,\|\cdot\|)$ be a $\mathcal{T}_{\varepsilon,\lambda}$-complete RN module over $\mathbb{K}$ with base $(\Omega,\mathcal{F},\mu)$. Then when $p>1$, for any $v \in L^p(S)$ there exists $L\in(L^p(S))_{m}^* (=L^q(S^*))$ such that $\|L\|^q=\|v\|^p=L(v)$, where $1<q<+\infty$ is such that $\frac{1}{p}+\frac{1}{q}=1$, and $\|L\|$ stands for the $L^0$-norm of $L$. When $p=1$, there exists $L\in L^{\infty}(S^*)$ such that $L(v)=\|v\|$ and $\|L\|=I_{[v\neq 0]}$.
\end{prop}
\begin{proof}
	By the Hahn-Banach theorem for random linear functions on the $RN$ module $(S,\|\cdot\|)$, see e.g. Example \ref{exm.2.12}, there exists $f\in S^{*}$ such that $f(v)=\|v\|$ and $\|f\|=I_{[v\neq 0]}$, where $[v\neq 0]$ stands for the equivalence class of the set $\{\omega\in \Omega:\|v\|^{0}(\omega)\neq 0\}$ for an arbitrarily chosen representative $\|v\|^{0}$ of $\|v\|$. Letting $L=\|v\|^{p-1}f$ completes the proof when $p>1$. When $p=1$, it is obvious.
\end{proof}

\begin{prop}\label{pro.3.13}
	Let $(S,\|\cdot\|)$ be a $\mathcal{T}_{\varepsilon,\lambda}$-complete RN module over $\mathbb{K}$ with base $(\Omega,\mathcal{F},\mu)$, $1\leq p\leq +\infty$ and $\mathcal{M}$ a submodule over $L^{\infty}(\mathcal{F},\mathbb{K})$ of $L^p(S)$. Then, for every $f\in \mathcal{M}_{m}^{*}$,  there exists $F\in (L^p(S))_{m}^*:=L^q(S^*)$ such that $F|_{\mathcal{M}}=f$ and $\normmm{F}_q=\normmm{f}:=\sup\{\int_{\Omega}|f(x)|d\mu : x\in \mathcal{M} ~$and$~\normmm{x}_p \leq 1\}$
\end{prop}
\begin{proof}
	We will divide the proof into two cases according to $p=+\infty$ or $1\leq p< +\infty$ as follows.
	
	When $p=+\infty$, since $\{x\in \mathcal{M}: \normmm{x}_{\infty}\leq 1\}=\{x\in \mathcal{M}: \|x\|\leq 1\}$, then $\normmm{f}:=\sup\{\int_{\Omega}|f(x)|d\mu : x\in \mathcal{M} ~$and$~\normmm{x}_{\infty} \leq 1\}=\sup\{\int_{\Omega}|f(x)|d\mu : x\in \mathcal{M} ~$and$~\|x\|\leq 1\}$. Now, let $\xi=\bigvee\{|f(x)|:x\in \mathcal{M}~\text{and}~\|x\|\leq 1\}$, then, similarly to the proof of Corollary \ref{coro.3.10} for $p=+\infty$, one can have that $\xi\in L_{+}^{1}(\mathcal{F})$, $|\xi|_{1}=\sup\{\int_{\Omega}|f(x)|d\mu:x\in \mathcal{M}~\text{and}~\|x\|\leq 1\}$ and $|f(x)|\leq \xi\|x\|$ for any $x\in \mathcal{M}$, which shows that $f\in (\mathcal{M},\|\cdot\|)^{*}$ and it is easy to see that $\|f\|=\xi$. By the Hahn-Banach theorem for an a.e. bounded random linear functional on an $RN$ space, there exists $F\in S^{*}$ such that $F\mid_{\mathcal{M}}=f$ and $\|F\|=\|f\|=\xi$, then $\normmm{F}_{1}=\normmm{f}=|\xi|_{1}<+\infty$, which also implies that $F\in L^{1}(S^{*})$.
	
	When $p<+\infty$, we can assume, without loss of generality, that $\mathcal{M}$ is a closed submodule (otherwise, we can consider the $\normmm{\cdot}_{p}$-closure of $\mathcal{M}$ and first extend $f$ to the closure), then it is very easy to check that $\mathcal{M}$ is an $L^{p}(\mathcal{F})$-normed $L^{\infty}(\mathcal{F},\mathbb{K})$-module, and hence by Theorem \ref{thm.3.6} there exists a $\mathcal{T}_{\varepsilon,\lambda}$-complete $RN$ module $(S_{1},\|\cdot\|)$ over $\mathbb{K}$ with base $(\Omega,\mathcal{F},\mu)$ such that $\mathcal{M}=L^{p}(S_{1})$ and $S_{1}$ must be a $\mathcal{T}_{\varepsilon,\lambda}$-closed $L^{0}(\mathcal{F},\mathbb{K})$-submodule of $S$, it immediately follows from Corollary \ref{coro.3.10} that there exists $F_{1}\in L^{q}(S_{1}^{*})$ such that $f(x)=F_{1}(x)$ for any $x\in \mathcal{M}$ and $\normmm{f}=\normmm{F_{1}}_{q}$. Further, $F_{1}$ has a random norm-preserving extension to $S$ (denoted by $F$), again by the Hahn-Banach theorem for an a.e. bounded random linear functional, namely $F\in S^{*}$ and $\|F\|=\|F_{1}\|$, then $\normmm{F}_{q}=\normmm{F_{1}}_{q}=\normmm{f}<+\infty$, namely $F\in L^q(S^{*})$, and $F$ is just desired.	
\end{proof}

Let us recall that a subset $G$ of an $L^{0}(\mathcal{F},\mathbb{K})$-module is said to $L^{0}$-convex if $\xi x+\eta y\in G$ for any $x$ and $y$ in $G$ and any $\xi$ and $\eta$ in $L_{+}^{0}(\mathcal{F})$ such that $\xi+\eta=1$. Let $(E,\mathcal{P})$ be a random locally convex module over $\mathbb{K}$ with base $(\Omega,\mathcal{F},\mu)$, $G$ a $\mathcal{T}_{\varepsilon,\lambda}$-closed nonempty subset of $S$ and $x\in S\backslash G$, then by Theorem 3.6 of \cite{Guo3} there exists $f\in E_{\varepsilon,\lambda}^{*}$ such that $(Ref)(x)>\bigvee\{(Ref)(y):y\in G\}$, where $(Ref)(z)=Re(f(z))$ stands for the real part of $f(z)$. Since an $L^{p}(\mathcal{F})$-normed $L^{\infty}(\mathcal{F},\mathbb{K})$-module $\mathcal{M}$ is a left module over $L^{\infty}(\mathcal{F},\mathbb{K})$, one can similarly have the notion of an $L^{0}$-convex subset in $\mathcal{M}$, we can also have the following statement:

\begin{prop}\label{pro.3.14}
	Let $(S,\|\cdot\|)$ be a $\mathcal{T}_{\varepsilon,\lambda}$-complete $RN$ module over $\mathbb{K}$ with base $(\Omega,\mathcal{F},\mu)$, $1\leq p\leq +\infty$, $G$ a nonempty $\normmm{\cdot}_{p}$-closed $L^{0}$-convex subset of $L^{p}(S)$ and $x\in L^{p}(S)\backslash G$. Then there exists $f\in (L^{p}(S))_{m}^{*}$ such that $(Ref)(x)>\bigvee\{(Ref)(g):g\in G\}$.
\end{prop}
\begin{proof}
	Denote $d(x,G)=\inf\{\normmm{x-g}_{p}:g\in G\}$, then $d(x,G)>0$ since $G$ is $\normmm{\cdot}_{p}$-closed and $x\notin G$. Further, denote $\xi=\bigwedge\{\|x-g\|:g\in G\}$, though the case here is different from that of Theorem 2.1 of \cite{GL}, by the method completely similar to that used in the proof of Theorem 2.1 of \cite{GL} one can prove that $d(x,G)=|\xi|_{p}$, so $\xi>0$. Now, let $G_{\varepsilon,\lambda}^{-}$ be the $\mathcal{T}_{\varepsilon,\lambda}$-closure of $G$ in $(S,\mathcal{T}_{\varepsilon,\lambda})$, then it is clear that $G_{\varepsilon,\lambda}^{-}$ is an $L^{0}$-convex subset of $S$ and  $\xi=\bigwedge\{\|x-y\|:y\in G_{\varepsilon,\lambda}^{-}\}$, so $x\notin G_{\varepsilon,\lambda}^{-}$, and hence by Theorem 3.6 of \cite{Guo3} there exists $F\in S^{*}$ such that $(ReF)(x)>\bigvee\{(ReF)(y):y\in G_{\varepsilon,\lambda}^{-}\}$, the latter is clearly equal to $\bigvee\{(ReF)(y):y\in G\}$. Since $\eta_{1}:=(ReF)(x)>\bigvee\{(ReF)(y):y\in G\}:=\eta_{2}$ means that $\eta_{1}\geq \eta_{2}$ but $\eta_{1}\neq \eta_{2}$, there exists $A\in \mathcal{F}$ with $\mu(A)>0$ such that $\eta_{1}>\eta_{2}$ on $A$. Besides, since $(\Omega,\mathcal{F},\mu)$ is $\sigma$-finite and $\|F\|\in L_{+}^{0}(\mathcal{F})$,we can choose a representative $\|F\|^{0}$ of $\|F\|$, sufficiently big $n\in \mathbb{N}$, and $\Omega_{0}\in \mathcal{F}$ with finite positive measure such that $\mu(A\cap(\|F\|^{0}\leq n)\cap \Omega_{0})> 0$, letting $A_{n}=A\cap(\|F\|^{0}\leq n)\cap \Omega_{0}$ and $f=\tilde{I}_{A_{n}}F$, then $f\in L^{q}(S^{*}):=(L^{p}(S))_{m}^{*}$ and still satisfies that $(Ref)(x)>\bigvee\{(Ref)(y):y\in G\}$.
\end{proof}

In the end of this section, we will show that Theorem 3.8 and Proposition 3.10 of \cite{LP2} can be regarded a special case of Corollary 3.10 of this paper. For this, let us first recall some necessary terminologies as follows. In this section, we always assume that $(B,\|\cdot\|)$ is a Banach space over $\mathbb{K}$.

A $\mathcal{F}$-measurable mapping $V:(\Omega,\mathcal{F},\mu)\rightarrow B$ is said to be simple if the range of $V$ is finite. A mapping $V:(\Omega,\mathcal{F},\mu)\rightarrow B$ is said to be strongly $\mathcal{F}$-measurable (strongly $\mu$-measurable) if it is a pointwise (a $\mu$-a.e. pointwise) limit of a sequence of simple functions, so a strongly $\mu$-measurable function is always equal $\mu$-a.e. to a strongly $\mathcal{F}$-measurable function.

A multivalued function $E:(\Omega,\mathcal{F},\mu)\rightarrow 2^{B}$ is said to be weakly measurable if $E^{-1}(G):=\{\omega\in \Omega:E(\omega)\cap G\neq \emptyset\}\in \mathcal{F}$ for any open set $G$ of $B$. A weakly measurable multivalued function $E:(\Omega,\mathcal{F},\mu)\rightarrow 2^{B}$ is called a measurable Banach bundle in \cite{LP2} if $E(\omega)$ is a closed subspace of $B$ for each $\omega\in \Omega$, and a measurable Banach bundle $E$ is further said to be reflexive if there exists a measurable set $A$ of measure zero such that $E(\omega)$ is reflexive for each $\omega\in \Omega\backslash A$.

Again, let us recall that a mapping $V:(\Omega,\mathcal{F},\mu)\rightarrow B'$ is said to be $w^{*}$-measurable if $(b,V(\cdot)):=V(\cdot)(b):(\Omega,\mathcal{F},\mu)\rightarrow \mathbb{K}$ is $\mathcal{F}$-measurable for each $b\in B$, which is equivalent to saying $(x(\cdot),V(\cdot)):(\Omega,\mathcal{F},\mu)\rightarrow \mathbb{K}$ is $\mathcal{F}$-measurable for each strongly $\mathcal{F}$-measurable function $x$ from $\Omega$ to $B$. Two $w^{*}$-measurable functions $V_{1}$ and $V_{2}:\Omega\rightarrow B'$ are said to be $w^{*}$-equivalent if $(b,V_{1}(\cdot))$ and $(b,V_{2}(\cdot))$ is equivalent for each $b\in B$. In \cite{LP2}, these terminologies were extended to a measurable Banach bundle $E$, let $\bar{\Gamma}_{0}(E)$ be the set of strongly $\mathcal{F}$-measurable sections (or selections) of $E$ and $E'$ be the multivalued function defined by $E'(\omega)=E(\omega)'$ for each $\omega\in \Omega$, a mapping $y:\Omega\rightarrow \cup_{\omega\in \Omega}E(\omega)'$ is called a $w^{*}$-measurable section of $E'$ if $y(\omega)\in E(\omega)'$ for each $\omega\in \Omega$ and $(x(\cdot),y(\cdot)):\Omega\rightarrow \mathbb{K}$ is $\mathcal{F}$-measurable for each $x\in \bar{\Gamma}_{0}(E)$. Furthermore, two $w^{*}$-measurable sections $y_{1}$ and $y_{2}$ of $E'$ are said to be $w^{*}$-equivalent if $(x(\cdot),y_{1}(\cdot))$ and $(x(\cdot),y_{2}(\cdot))$ are equivalent for each $x\in \bar{\Gamma}_{0}(E)$. Finally, we denote by $\bar{\Gamma}_{0}(E'_{w^{*}})$ the set of $w^{*}$-measurable sections of $E'$.

To state Lemmas \ref{lem.3.17} and \ref{lem.3.19} below, we first give Examples \ref{ex.3.15} and \ref{ex.3.16} below for the clarity.

\begin{ex}\label{ex.3.15}
	Denote by $L^{0}(\mathcal{F},B)$ the set of equivalence classes of strongly $\mathcal{F}$-measurable functions from $(\Omega,\mathcal{F},\mu)$ to $(B,\|\cdot\|)$, then $L^{0}(\mathcal{F},B)$ forms a $\mathcal{T}_{\varepsilon,\lambda}$-complete $RN$ module over $\mathbb{K}$ with base $(\Omega,\mathcal{F},\mu)$ in a natural way, see \cite[Example 2.3]{Guo3}. Further, let $\Gamma_{0}(E)$ be the set of equivalence classes of strongly $\mathcal{F}$-measurable sections of a measurable Banach bundle $E$ from $(\Omega,\mathcal{F},\mu)$ to $B$, then, similar to the construction of $L^{0}(\mathcal{F},B)$, $\Gamma_{0}(E)$ is a $\mathcal{T}_{\varepsilon,\lambda}$-complete $RN$ module and a submodule of $L^{0}(\mathcal{F},B)$.
\end{ex}

\begin{ex}\label{ex.3.16}
	Let $L^{0}(\mathcal{F},B',w^{*})$ be the set of $w^{*}$-equivalence classes of $w^{*}$-measurable functions from $(\Omega,\mathcal{F},\mu)$ to $B'$, then, similar to the construction of $L^{0}(\mu,B',w^{*})$ in \cite{Guo7a,Guo5}, $L^{0}(\mathcal{F},B',w^{*})$ forms a $\mathcal{T}_{\varepsilon,\lambda}$-complete $RN$ module over $\mathbb{K}$ with base $(\Omega,\mathcal{F},\mu)$. Further, let $\Gamma_{0}(E'_{w^{*}})$ be the set of $w^{*}$-equivalence classes of $w^{*}$-measurable sections of $E'$, where $E$ is a measurable Banach bundle from  $(\Omega,\mathcal{F},\mu)$ to $B$, it is clear that $\Gamma_{0}(E'_{w^{*}})$ forms an $L^{0}(\mathcal{F},\mathbb{K})$-module in the same way as $L^{0}(\mathcal{F},B',w^{*})$. Let $g$ be any element of $\Gamma_{0}(E'_{w^{*}})$ with an arbitrarily chosen representative $g^{0}\in \bar{\Gamma}_{0}(E'_{w^{*}})$, denote by $\xi_{g}^{0}$ the essential supremum of the set $\{|(v(\cdot),g^{0}(\cdot))|:v\in \bar{\Gamma}_{0}(E) ~\text{and} ~\|v(\omega)\|\leq 1 ~\text{for each }~ \omega\in \Omega\}$ ($\xi_{g}^{0}$ is well defined and can be chosen as a real-valued $\mathcal{F}$-measurable function, see \cite[pp.3026-3027]{Guo3}), then we define the $L^{0}$-norm $\|g\|$ of $g$ as the equivalence class of $\xi_{g}^{0}$, then $(\Gamma_{0}(E'_{w^{*}}),\|\cdot\|)$ becomes a $\mathcal{T}_{\varepsilon,\lambda}$-complete $RN$ module over $\mathbb{K}$ with base $(\Omega,\mathcal{F},\mu)$.
\end{ex}

Although the completeness of $\sigma$-algebra $\mathcal{F}$ is needed when the theory of the lifting property in \cite{II} is applied to the Banach algebra $L^{\infty}(\mathcal{F},\mathbb{K})$, the terminology ``strongly $\mu$-measurable functions'' was employed in \cite{Guo7a,Guo5,LP2} and since a strongly $\mu$-measurable function is always $\mathcal{F}^{\mu}$-measurable (where $\mathcal{F}^{\mu}$ is the completion of $\mathcal{F}$ with respect to $\mu$), it is not necessary to assume that the $\sigma$-finite measure space involved in the work of \cite{Guo7a,Guo5,LP2} is complete since the theory of the lifting property can be applied to $L^{\infty}(\mathcal{F}^{\mu},\mathbb{K})$. If the terminology ``strongly $\mathcal{F}$-measurable functions'' is employed in \cite{Guo7a,Guo5}  as in this paper, then, in fact, Guo already proved the following representation theorem of the random conjugate space of $L^{0}(\mathcal{F},B)$ in \cite{Guo7a,Guo5}: if $(\Omega,\mathcal{F},\mu)$ is complete, then $L^{0}(\mathcal{F},B',w^{*})$ is isometrically isomorphic to $L^{0}(\mathcal{F},B)^{*}$ under the canonical mapping $T$, where, for any $f\in L^{0}(\mathcal{F},B',w^{*})$, $T_{f}:=T(f):L^{0}(\mathcal{F},B)\rightarrow L^{0}(\mathcal{F},\mathbb{K})$ is defined by $T_{f}(g)=f(g):=(g,f)$ for any $g\in L^{0}(\mathcal{F},B)$, and $(g,f)$ stands for the equivalence class of $(g^{0}(\cdot),f^{0}(\cdot))$, where $f^{0}$ and $g^{0}$ are arbitrarily chosen representatives of $f$ and $g$, respectively. Similarly, we can have Lemma \ref{lem.3.17} below, which was essentially used in the proof of Theorem 3.8 of \cite{LP2}.	

\begin{lem}\label{lem.3.17}
	Suppose $(\Omega,\mathcal{F},\mu)$ is complete, and further let $\Gamma_{0}(E)$ and $\Gamma_{0}(E'_{w^{*}})$ be the same as in Examples \ref{ex.3.15} and \ref{ex.3.16}, then $\Gamma_{0}(E'_{w^{*}})$ is isometrically isomorphic to $\Gamma_{0}(E)^{*}$ under the canonical mapping $T$, where, for any $f\in \Gamma_{0}(E'_{w^{*}})$, $T_{f}:=T(f):\Gamma_{0}(E)\rightarrow L^{0}(\mathcal{F},\mathbb{K})$ can be similarly defined as in the paragraph before this lemma.
\end{lem}
\begin{proof}
	For any $f\in \Gamma_{0}(E'_{w^{*}})$, it is obvious that $|T_{f}(g)|\leq \|f\|\|g\|$ for any $g\in \Gamma_{0}(E)$, and thus $T_{f}\in \Gamma_{0}(E)^{*}$ and $\|T_{f}\|\leq \|f\|$, while $\|T_{f}\|=\bigvee\{|T_{f}(g)|:g\in \Gamma_{0}(E)~\text{and}~\|g\|\leq 1\}=\bigvee\{|(g,f)|:g\in \Gamma_{0}(E)~\text{and}~\|g\|\leq 1\}$, so $\|T_{f}\|=\|f\|$ follows from the definition of $\|f\|$.
	
	Further, let $F$ be any element of $\Gamma_{0}(E)^{*}$, then by the Hahn-Banach theorem of random linear functional, there exists $\tilde{F}\in L^{0}(\mathcal{F},B)^{*}$ such that $\tilde{F}|_{\Gamma_{0}(E)}=F$ and $\|\tilde{F}\|=\|F\|$, and thus there exists $\tilde{f}\in L^{0}(\mathcal{F},B',w^{*})$ such that $\tilde{F}(g)=(g,\tilde{f})$ for any $g\in L^{0}(\mathcal{F},B)$ and $\|\tilde{F}\|=\|\tilde{f}\|$. Now, let $\tilde{f}^{0}$ be an arbitrarily chosen representative of $\tilde{f}$, for each $\omega\in \Omega$, denote by $f^{0}(\omega)$ the restriction of $\tilde{f}^{0}(\omega)$ to $E(\omega)$, then it is easy to verify that $f^{0}\in \bar{\Gamma}_{0}(E'_{w^{*}}$ and $F=T_{f}$, where $f$ is the $w^{*}$-equivalence of $f^{0}$, which shows that $T$ is also surjective.	
\end{proof}

Corollary \ref{coro.3.18} below amounts to Theorem 3.8 of \cite{LP2}.

\begin{cor}\label{coro.3.18}
	Suppose $(\Omega,\mathcal{F},\mu)$ is complete and $p\in [1,+\infty]$, then $(\Gamma_{p}(E))_{m}^{*}\cong \Gamma_{q}(E'_{w^{*}})$, where $q$ is the H\"{o}lder conjugate number of $p$, $\Gamma_{p}(E)(\Gamma_{q}(E'_{w^{*}}))$ is the $L^{p}(\mathcal{F})$-normed ($L^{q}(\mathcal{F})$-normed) $L^{\infty}(\mathcal{F},\mathbb{K})$-module generated by $\Gamma_{0}(E)(\Gamma_{0}(E'_{w^{*}}))$.
\end{cor}
\begin{proof}
	It immediately follows from Corollary \ref{coro.3.10} by taking $S=\Gamma_{0}(E)$.
\end{proof}

Lemma \ref{lem.3.19} below both is of independent interest and can be used to considerably simplify the proof of Proposition 3.10 of \cite{LP2}.

\begin{lem}\label{lem.3.19}
	Let $(B,\|\cdot\|)$ be a separable Banach space and $E$ be a measurable reflexive Banach bundle from $(\Omega,\mathcal{F},\mu)$ to $B$, then $\Gamma_{0}(E)$ is random reflexive.
\end{lem}
\begin{proof}
	Since each strongly $\mathcal{F}^{\mu}$-measurable function is equivalent to a strongly $\mathcal{F}$-measurable function (where $\mathcal{F}^{\mu}$ is the completion of $\mathcal{F}$ with respect to $\mu$), $\Gamma_{0}(E)$ can be identified with the set of equivalence classes of strongly $\mathcal{F}^{\mu}$-measurable sections of $E$, we can, without loss of generality, assume that $(\Omega,\mathcal{F},\mu)$ itself is complete. By $(g)$ of Theorem 4.2 of \cite{Wagner}, there is a sequence $\{V_{n}^{0},n\in \mathbb{N}\}$ of $\mathcal{F}$-measurable sections of $E$ such that $cl(\{V_{n}^{0}(\omega),n\in \mathbb{N}\})=E(\omega)$ for each $\omega\in E$, where $cl(G)$ stands for the closure of $G$ for any subset $G$ of $B$. It is obvious that each $V_{n}^{0}$ is also strongly $\mathcal{F}$-measurable since $B$ is separable. Now, we define $\bar{V}_{n}^{0}:\Omega\rightarrow B$ by $\bar{V}_{n}^{0}(\omega)=I_{(\|V_{n}^{0}\|\leq 1)}(\omega)V_{n}^{0}(\omega)$ for each $\omega\in \Omega$, where $I_{(\|V_{n}^{0}\|\leq 1)}$ stands for the characteristic function of the set $(\|V_{n}^{0}\|\leq 1):=\{\omega\in \Omega:\|V_{n}^{0}\|(\omega)\leq 1\}$. Then each $\bar{V}_{n}^{0}$ is still a strongly $\mathcal{F}$-measurable section of $E$ since each $E(\omega)$ is a subspace of $B$, and moreover, one can also easily verify that the following holds:
	\begin{equation}\label{(3.1)}
		cl(\{\bar{V}_{n}^{0}(\omega),n\in \mathbb{N}\})=\{b\in B:b\in E(\omega)~\text{and}~\|b\|\leq 1\}~\text{for each}~\omega\in \Omega.
	\end{equation}

    By the James theorem of random reflexivity \cite{GL}, it suffices to prove that each element $f\in \Gamma_{0}(E)^{*}$ can attain its $L^{0}$-norm on the random closed unit ball $\{x\in \Gamma_{0}(E):\|x\|\leq 1\}$. Now, let $f$ be any given element of $\Gamma_{0}(E)^{*}$, then by Lemma \ref{lem.3.17} there exists an unique $y\in \Gamma_{0}(E'_{w^{*}})$ such that $f(x)=(x,y)$ for any $x\in \Gamma_{0}(E)$ and $\|f\|=\|y\|$. Further, let $y^{0}\in \bar{\Gamma}(E'_{w^{*}})$ be an arbitrarily chosen representative of $y$, then one can easily see that $y^{0}(\omega)\in E(\omega)'$ for each $\omega\in \Omega$ (by the definition of $\bar{\Gamma}_{0}(E'_{w^{*}})$), and it immediately follows from (\ref{(3.1)}) that $\|y^{0}(\omega)\|=\sup\{|y^{0}(\omega)(\bar{V}_{n}^{0}(\omega))|:n\in \mathbb{N}\}$ for each $\omega\in \Omega$, from which $\|f\|$ is exactly the equivalence of $\|y^{0}(\cdot)\|$. Next, we continue to look for an element $x\in \Gamma_{0}(E)$ such that $\|x\|\leq 1$ and $|f(x)|=\|f\|$ as follows.

    Since $E$ is reflexive, we can, without loss of generality, assume that $E(\omega)$ is reflexive for each $\omega\in \Omega$, then, by the classical James theorem \cite{James}, there exists $b\in E(\omega)$ for each $\omega\in \Omega$ such that $\|b\|\leq 1$ and $|y^{0}(\omega)(b)|=\|y^{0}(\omega)\|$, namely $Gr(E)\cap \{(\omega,b)\in \Omega\times\{b\in B:\|b\|\leq 1\}:|y^{0}(\omega)(b)|=\|y^{0}(\omega)\|\}\neq \emptyset$, where $Gr(E)$ stands for the graph of $E$. From the process of obtaining $y$ in the proof of Lemma \ref{lem.3.17}, one can see that there exists a $w^{*}$-measurable function $\tilde{y}^{0}:\Omega\rightarrow B'$ such that $\|y^{0}(\omega)\|=\|\tilde{y}^{0}(\omega)\|$ and $y^{0}(\omega)=\tilde{y}^{0}(\omega)|_{E(\omega)}$ for each $\omega\in \Omega$, then $Gr(E)\cap \{(\omega,b)\in \Omega\times\{b\in B:\|b\|\leq 1\}:|\tilde{y}^{0}(\omega)(b)|=\|\tilde{y}^{0}(\omega)\|\}\neq \emptyset$. Now, define a multivalued function $E_{1}:\Omega\rightarrow 2^{B}$ as follows:
    \begin{equation*}
    	E_{1}(\omega)=\{b\in B:b\in E(\omega),\|b\|\leq 1~\text{and}~|\tilde{y}^{0}(\omega)(b)|=\|\tilde{y}^{0}(\omega)\|\}\text{for each}~\omega\in \Omega.
    \end{equation*}
    Then, since $\|\tilde{y}^{0}(\cdot)\|=\|y^{0}(\cdot)\|$ is $\mathcal{F}$-measurable and $Gr(E_{1})=Gr(E)\cap \{(\omega,b)\in \Omega\times \{b\in B:\|b\|\leq 1\}:|\tilde{y}^{0}(\omega)(b)|=\|\tilde{y}^{0}(\omega)\|\}$, it is clear that $E_{1}$ is closed-valued and has a $\mathcal{F}\otimes \mathcal{B}(B)$-measurable graph (where $\mathcal{B}(B)$ stands for the Borel $\sigma$-algebra of $B$), then, again by $(g)$ of Theorem 4.2 of \cite{Wagner}, there exists an $\mathcal{F}$-measurable section $x^{0}$ of $E_{1}$, which is also a section of $E$, such that $\|x^{0}(\omega)\|\leq 1$ and $\|y^{0}(\omega)\|=\|\tilde{y}^{0}(\omega)\|=\|\tilde{y}^{0}(\omega)(x^{0}(\omega))\|=\|y^{0}(\omega)(x^{0}(\omega))\|$ for each $\omega\in \Omega$. Finally, let $x$ be the equivalence class of $x^{0}$, then $x\in \Gamma_{0}(E)$, $\|x\|\leq 1$ and $\|f\|=\|y\|=|(x,y)|=f(x)$.
\end{proof}

Corollary \ref{coro.3.20} below amounts to Proposition 3.10 of \cite{LP2}.

\begin{cor}\label{coro.3.20}
	Let $(\Omega,\mathcal{F},\mu)$ be complete, $(B,\|\cdot\|)$ a separable Banach space over $\mathbb{K}$ and $E$ a reflexive measurable Banach bundle from $(\Omega,\mathcal{F},\mu)$ to $B$, then $\Gamma_{p}(E)\cong (\Gamma_{q}(E'_{w^{}}))_{m}^{*}$ for any $p$ and $q\in [1,\infty]$ such that $\frac{1}{p}+\frac{1}{q}=1$.
\end{cor}
\begin{proof}
	Since $\Gamma_{0}(E)$ is random reflexivity by Lemma \ref{lem.3.19}, $\Gamma_{0}(E)=\Gamma_{0}(E)^{**}=((\Gamma_{0}(E))^{*})^{*}=(\Gamma_{0}(E'_{w^{*}}))^{*}$ by Lemma \ref{lem.3.17}. Further, by noticing $\Gamma_{p}(E)=L^{p}(\Gamma_{0}(E))$ and $\Gamma_{q}(E'_{w^{*}})=L^{q}(E'_{w^{*}})$, then, by Corollary \ref{coro.3.10}, one can have that $\Gamma_{p}(E)=L^{p}(\Gamma_{0}(E))=L^{p}(\Gamma_{0}(E)^{**})=L^{p}((\Gamma_{0}(E'_{w^{*}}))^{*})\cong (L^{q}(E'_{w^{*}}))_{m}^{*}$.
\end{proof}

\begin{rem}\label{rem.3.21}
	In Corollaries \ref{coro.3.18} and \ref{coro.3.20}, we allow that $p$ and $q$ are any pair of H\"{o}lder conjugate numbers, so the restriction that $p$ and $q$ belong to $(1,+\infty)$ in Theorem 3.8 and Proposition 3.10 of \cite{LP2} is not necessary. When $1<p<+\infty$, Corollary \ref{coro.3.11} shows that $L^{p}(S)$ is module reflexive iff $S$ is random reflexive, and thus the truly excellent and also most difficult part in Section 3 of \cite{LP2} is Theorem 3.12 of \cite{LP2}, which shows that the converse of Lemma \ref{lem.3.19} still holds, namely the random reflexivity of $\Gamma_{0}(E)$ also implies the reflexivity of $E$.
\end{rem}

\begin{rem}\label{rem.3.22}
	Just as pointed out in the beginning of the proof Lemma \ref{lem.3.19}, $L^{0}(\mathcal{F},B)$ and $L^{0}(\mathcal{F}^{\mu},B)$ can be identified as the spaces of equivalence classes of strongly measurable functions, but we do not know if $L^{0}(\mathcal{F},B',w^{*})$ can be identified with $L^{0}(\mathcal{F}^{\mu},B',w^{*})$ since they are two complicated objects. The embedding mapping $i:L^{0}(\mathcal{F}^{\mu},B')\rightarrow L^{0}(\mathcal{F}^{\mu},B',w^{*})$ defined by $i(x)=$ the $w^{*}$-equivalence class of $x^{0}$ for any $x\in L^{0}(\mathcal{F}^{\mu},B)$ with a representative $x^{0}$ is obviously an isometric embedding since $x^{0}$ has a separable range. It is proved in \cite{Guo7a} that $L^{0}(\mathcal{F}^{\mu},B')\cong L^{0}(\mathcal{F}^{\mu},B)^{*}$ under the canonical mapping iff $i$ is surjective iff each $w^{*}$-$\mathcal{F}^{\mu}$-measurable function is $w^{*}$-equivalent to a strongly $\mathcal{F}^{\mu}$-measurable function iff $B'$ has the Radon-Nikod\'{y}m property with respect to $(\Omega,\mathcal{F},\mu)$, so we also have that $L^{0}(\mathcal{F},B')\cong L^{0}(\mathcal{F},B)^{*}$ iff $B'$ has the Radon-Nikod\'{y}m property with respect to $(\Omega,\mathcal{F},\mu)$, namely the completeness of $(\Omega,\mathcal{F},\mu)$ is not necessary when the propositions only involve the spaces of equivalence classes of strongly measurable functions, see \cite{DU} for the survey of the Radon-Nikod\'{y}m property.
\end{rem}

%###################################################################################################
%###################################################################################################
\section{On $d$-decomposability and $d$-$\sigma$-stability in order complete $RN$ spaces}
%###################################################################################################
%###################################################################################################

The aim of this section is to prove that the two notions of $d$-decomposability and $d$-$\sigma$-stability coincide for an order complete $RN$ space, the main results of this section are Theorems \ref{thm.4.2} and \ref{thm.4.4}.

$RN$ spaces are a special class of lattice-normed spaces, let us first recall the related terminologies of lattice-normed spaces. Definition \ref{def.4.1} below is taken from \cite{KK}.

\begin{defn}\label{def.4.1}
	An ordered tripe $(S,\|\cdot\|,E)$ is called a lattice-normed space over $\mathbb{K}$ with an $E$-valued norm $\|\cdot\|$ if $S$ is a linear space over $\mathbb{K}$, $E$ is a (real) vector lattice and $\|\cdot\|$ is a mapping from $S$ to $E_{+}$ (where $E_{+}=\{e\in E:e\geq 0\}$) such that the following axioms are satisfied:\\
	(LN-1) $\|x\|=0$ iff $x=\theta$;\\
	(LN-2) $\|\lambda x\|=|\lambda|\|x\|$ for any $x\in S$ and any $\lambda\in \mathbb{K}$;\\
	(LN-3) $\|x+y\|\leq \|x\|+\|y\|$ for any $x$ and $y$ in $S$.	
%	\begin{itemize}
%		\item [(LN-1)] $\|x\|=0$ iff $x=\theta$;
%		\item [(LN-2)] $\|\lambda x\|=|\lambda|\|x\|$ for any $x\in S$ and any $\lambda\in \mathbb{K}$;
%		\item [(LN-3)] $\|x+y\|\leq \|x\|+\|y\|$ for any $x$ and $y$ in $S$.	
%	\end{itemize}
\end{defn}

Let $(S,\|\cdot\|,E)$ be a lattice-normed space. $\|\cdot\|$ is said to be $d$-decomposable if, for any $x\in S$ and any $e_{1}$ and $e_{2}$ in $E$ with $e_{1}\wedge e_{2}=0$ such that $\|x\|=e_{1}+e_{2}$, there exist $x_{1}$ and $x_{2}$ in $S$ such that $x=x_{1}+x_{2}$, $\|x_{1}\|=e_{1}$ and $\|x_{2}\|=e_{2}$, at this time $\|\cdot\|$ is also called a Kantorovich norm and $(S,\|\cdot\|,E)$ is also said to be $d$-decomposable. A net $\{x_{\alpha},\alpha\in \Lambda\}$ in $S$ is said to be convergent in order (briefly, order convergent or $o$-convergent) to an element $x$ in $S$ and written as $x=o-\lim_{\alpha}x_{\alpha}$ if there exists a decreasing (or nonincreasing) net $\{e_{\gamma},\gamma\in \Gamma\}$ in $E$ such that $\wedge_{\gamma\in \Gamma}e_{\gamma}=0$ and, for every $\gamma\in \Gamma$, there exists an $\alpha(\gamma)\in \Lambda$ such that $\|x-x_{\alpha}\|\leq e_{\gamma}$ whenever $\alpha\geq \alpha(\gamma)$. A net $\{x_{\alpha},\alpha\in \Lambda\}$ in $S$ is said to be Cauchy (or fundamental) in order (briefly, $o$-Cauchy) if the net $\{x_{\alpha}-x_{\beta},(\alpha,\beta)\in \Lambda\times \Lambda\}~o$-converges to $\theta$ (the null element in $S$). $(S,\|\cdot\|,E)$ is said to be complete in order (briefly, order complete or $o$-complete) if every $o$-Cauchy net in $S$ is $o$-convergent.

Clearly, an $RN$ space $(S,\|\cdot\|)$ over $\mathbb{K}$ with base $(\Omega,\mathcal{F},\mu)$ is a lattice-normed space $(S,\|\cdot\|,E)$ with $E=L^{0}(\mathcal{F})$. Since an $RN$ space has an additional measure-theoretic construct, we can consider the $(\varepsilon,\lambda)$-linear topology, $\mathcal{T}_{\varepsilon,\lambda}$-completeness and other issues closely related to measure theory.

\begin{thm}\label{thm.4.2}
	An $RN$ space $(S,\|\cdot\|)$ over $\mathbb{K}$ with base $(\Omega,\mathcal{F},\mu)$ is $o$-complete iff it is $\mathcal{T}_{\varepsilon,\lambda}$-complete.
\end{thm}
\begin{proof}
	Sufficiency. Let $\{x_{\alpha},\alpha\in \Lambda\}$ be an $o$-Cauchy net in $S$, then there exists a decreasing net $\{e_{\gamma},\gamma\in \Gamma\}$ in $L_{+}^{0}(\mathcal{F})$ with $\wedge_{\gamma\in \Gamma}e_{\gamma}=0$ such that, for each $\gamma\in \Gamma$, there exists $\alpha(\gamma)\in \Lambda$ satisfying $\|x_{\alpha}-x_{\beta}\|\leq e_{\gamma}$ whenever $\alpha,\beta\geq\alpha(\gamma)$. Since $\{e_{\gamma},\gamma\in \Gamma\}$ is decreasing, there exists a decreasing sequence $\{e_{\gamma_{n}},n\in \mathbb{N}\}$ in $\{e_{\gamma},\gamma\in \Gamma\}$ such that $\{e_{\gamma_{n}},n\in \mathbb{N}\}$ converges a.e. to $0$ by Proposition \ref{pro.1.1}, and hence $\{e_{\gamma_{n}},n\in \mathbb{N}\}$  also converges in probability measure $P_{\mu}$ to $0$, then, for any $\varepsilon>0$ and $0<\lambda<1$, there exists $n_{0}(\varepsilon,\lambda)\in \mathbb{N}$ such that $P_{\mu}\{\omega\in \Omega:e_{\gamma_{n}}(\omega)<\varepsilon\}>1-\lambda$ whenever $n\geq n_{0}(\varepsilon,\lambda)$. Taking $\alpha_{0}=\alpha(\gamma_{n_{0}(\varepsilon,\lambda)})$ yields that $P_{\mu}\{\omega\in \Omega:\|x_{\alpha}-x_{\beta}\|(\omega)<\varepsilon\}\geq P_{\mu}\{\omega\in \Omega:e_{\gamma_{n_{0}(\varepsilon,\lambda)}}(\omega)<\varepsilon\}>1-\lambda$ whenever $\alpha,\beta\geq \alpha_{0}$, which means that $\{x_{\alpha},\alpha\in \Lambda\}$ is a $\mathcal{T}_{\varepsilon,\lambda}$-Cauchy net in $S$, and hence convergent in $\mathcal{T}_{\varepsilon,\lambda}$ to some element $x$ in $S$ by $\mathcal{T}_{\varepsilon,\lambda}$-completeness of $S$. We also have that $\|x_{\alpha}-x\|\leq e_{\gamma}$ for any $\alpha\geq \alpha(\gamma)$ since $\{z\in S:\|x_{\alpha}-z\|\leq e_{\gamma}\}$ is $\mathcal{T}_{\varepsilon,\lambda}$-closed and $\{x_{\beta},\beta\geq \alpha(\gamma)\}$  is a subnet of $\{x_{\beta},\beta\in \Lambda\}$ for any given $\gamma\in \Gamma$ and any given $\alpha\geq \alpha(\gamma)$, which implies that $\{x_{\alpha},\alpha\in \Lambda\}$ converges in order to $x$.
	
	Necessity. Since $\mathcal{T}_{\varepsilon,\lambda}$ is metrizable, we only need to consider $\mathcal{T}_{\varepsilon,\lambda}$-Cauchy sequences for the proof of Necessity. Let $\{x_{n},n\in \mathbb{N}\}$ be a $\mathcal{T}_{\varepsilon,\lambda}$-Cauchy sequence in $S$, then there exists a subsequence $\{x_{n_{k}},k\in \mathbb{N}\}$ of $\{x_{n},n\in \mathbb{N}\}$ such that $\{\|x_{n_{k}}-x_{n_{l}}\|,k,l\in \mathbb{N}\}$ converges a.e. (with respect to $P_{\mu}$, equivalently also with respect to $\mu$) to $0$ when $k,l\rightarrow \infty$. Let $e_{m}=\bigvee\{\|x_{n_{k}}-x_{n_{l}}\|:k,l\geq m\}$ for each $m\in \mathbb{N} $, then it is obvious that $\{e_{m},m\in \mathbb{N}\}$ is a decreasing sequence in $L_{+}^{0}(\mathcal{F})$ such that $\wedge_{m}e_{m}=0$ and, for every $m\in \mathbb{N}$, it is also clear that $\|x_{n_{k}}-x_{n_{l}}\|\leq e_{m}$ for any $k,l\geq m$, namely $\{x_{n_{k}},k\in \mathbb{N}\}$ is an $o$-Cauchy sequence in $S$, and hence convergent in order to some $x$ in $S$ by $o$-completeness of $S$. By the similar argument used in the proof of sufficiency, one can see that $\{x_{n_{k}},k\in \mathbb{N}\}$ also converges in $\mathcal{T}_{\varepsilon,\lambda}$ to $x$, which further means that $\{x_{n},n\in \mathbb{N}\}$ converges in $\mathcal{T}_{\varepsilon,\lambda}$ to $x$ since it is $\mathcal{T}_{\varepsilon,\lambda}$-Cauchy.
\end{proof}

To prove the second main result of this section--Theorem \ref{thm.4.4} below, we first give Lemma \ref{lem.4.3} below.

\begin{lem}\label{lem.4.3}
	An $RN$ space $(S,\|\cdot\|)$ over $\mathbb{K}$ with base $(\Omega,\mathcal{F},\mu)$ is $d$-decomposable iff it is $d$-stable.
\end{lem}
\begin{proof}
	Necessity. For any given $x$ in $S$ and any given $a\in B_{\mathcal{F}}$, let $e_{1}=I_{a}\|x\|$ and $e_{2}=I_{a^{c}}\|x\|$, then $\|x\|=e_{1}+e_{2}$ and $e_{1}\wedge e_{2}=0$, and thus there exist two elements $x_{1}$ and $x_{2}$ in $S$ such that $x=x_{1}+x_{2}$, $\|x_{1}\|=e_{1}$ and $\|x_{2}\|=e_{2}$ by $d$-decomposability of $\|\cdot\|$, where the uniqueness of $x_{1}$ and $x_{2}$ are obvious, we denote $x_{1}$ so obtained by $I_{a}x$. Further, for any $x$ and $y$ in $S$ and any $a\in B_{\mathcal{F}}$, then $z=I_{a}x+I_{a^{c}}y$ obviously satisfies $I_{a}\|z-x\|=0$ and $I_{a^{c}}\|z-y\|=0$ since $\|x-I_{a}x\|=I_{a^{c}}\|x\|$ and $\|y-I_{a^{c}}y\|=I_{a}\|y\|$.
	
	Sufficiency. Let $\|x\|=e_{1}+e_{2}$ be such that $e_{1}\wedge e_{2}=0$, arbitrarily choose $e_{1}^{0}$ and $e_{2}^{0}$ as a representative of $e_{1}$ and $e_{2}$, respectively, and further let $a=[e_{1}\neq 0]$ be the equivalence class of $\{\omega\in \Omega:e_{1}^{0}(\omega)\neq 0\}$, then it is clear that $I_{a}e_{1}=e_{1}$ and $I_{a}e_{2}=0$, and hence $e_{2}=(I_{a}+I_{a^{c}})e_{2}=I_{a^{c}}e_{2}$ and $I_{a^{c}}e_{1}=e_{1}-I_{a}e_{1}=0$. Let $x_{1}=I_{a}x+I_{a^{c}}\theta$ and $x_{2}=I_{a^{c}}x+I_{a}\theta$, which indeed exist by $d$-stability of $S$, then, by the definition of $d$-stability, $I_{a}\|x-x_{1}\|=0,~I_{a^{c}}\|x_{1}-\theta\|=0,~I_{a^{c}}\|x-x_{2}\|=0$ and $I_{a}\|x-\theta\|=0$, and thus $\|x_{1}\|=I_{a}\|x_{1}\|+I_{a^{c}}\|x_{1}\|=I_{a}\|x\|$ by the triangle inequality (RN-3), similarly, we can also have $\|x_{2}\|=I_{a^{c}}\|x\|$. Again by the definition of $d$-stability, we have that $x=I_{a}x+I_{a^{c}}x$ and $\theta=I_{a}\theta+I_{a^{c}}\theta$, which shows $x=x_{1}+x_{2}$.
\end{proof}

\begin{thm}\label{thm.4.4}
	Let $(S,\|\cdot\|)$ be an order complete $RN$ space over $\mathbb{K}$ with base $(\Omega,\mathcal{F},\mu)$. Then $\|\cdot\|$ is $d$-decomposable iff $S$ is $d$-$\sigma$-stable, and in which case $(S,\|\cdot\|)$ becomes a $\mathcal{T}_{\varepsilon,\lambda}$-complete $RN$ module in a nature way.
\end{thm}
\begin{proof}
	Lemma \ref{lem.4.3} shows that $\|\cdot\|$ is $d$-decomposable iff $S$ is $d$-stable, Theorem \ref{thm.4.2} shows that $(S,\|\cdot\|)$ is order complete iff it is $\mathcal{T}_{\varepsilon,\lambda}$-complete, and while Theorem 2.15 of \cite{GWYZ} shows that $d$-stability and $d$-$\sigma$-stability coincide for an $(\varepsilon,\lambda)$-complete $RM$ space, then $\|\cdot\|$ is $d$-decomposable iff $S$ is $d$-$\sigma$-stable by applying Theorem 2.15 of \cite{GWYZ} to the $RM$ space $(S,d)$ with $d$ induced by $\|\cdot\|$.
	
	When $\|\cdot\|$ is $d$-decomposable, namely $(S,\|\cdot\|)$ is $d$-stable, for any simple element $\varphi$ in $L^{0}(\mathcal{F},\mathbb{K})$, written as $\varphi=\sum_{i=1}^{m}\alpha_{i}I_{a_{i}}$, where $\alpha_{i}\in \mathbb{K}$ for any $i=1\sim m$ and $\{a_{i},i=1\sim m\}$ forms a finite partition of unity in $B_{\mathcal{F}}$, and for any $x\in S$, define the module multiplication $\varphi x$ of $\varphi$ and $x$ by $\varphi x=\sum_{i=1}^{m}I_{a_{i}}(\alpha_{i} x)$, it is easy to check that $\|\varphi x\|=|\varphi|\|x\|$, and it is also very easy to verify that the definition of $\varphi x$ is independent of a particular representation of $\varphi$ whenever $\{a_{i},i=1\sim m\}$ forms a finite partition of unity. Now, for an arbitrary $\xi\in L^{0}(\mathcal{F},\mathbb{K})$, we can first choose a sequence $\{\varphi_{n},n\in \mathbb{N}\}$ of simple elements in $L^{0}(\mathcal{F},\mathbb{K})$ such that $|\varphi_{n}|\leq |\xi|$ for each $n\in \mathbb{N}$ and $\{\varphi_{n},n\in \mathbb{N}\}$ converges a.e. to $\xi$, and then define the module multiplication $\xi x$ of $\xi$ and $x$ as the $\mathcal{T}_{\varepsilon,\lambda}$-limit of $\{\varphi_{n}x,n\in \mathbb{N}\}$, which is well defined and $\|\xi x\|=|\xi|\|x\|$ since it is easy to check that $\{\varphi_{n}x,n\in \mathbb{N}\}$ is a $\mathcal{T}_{\varepsilon,\lambda}$-Cauchy sequence and further by the fact that $S$ is $\mathcal{T}_{\varepsilon,\lambda}$-complete (again, it is also obvious that the definition of $\xi x$ is independent of a particular choice of $\{\varphi_{n},n\in \mathbb{N}\}$). Finally, it is easy to observe that $(S,\|\cdot\|)$ becomes an $RN$ module under the module multiplication stated above.
\end{proof}

\begin{rem}\label{rem.4.5}
	Theorem \ref{thm.4.4} shows that an order complete $d$-decomposable $RN$ space is a $\mathcal{T}_{\varepsilon,\lambda}$-complete $RN$ module, and hence its $d$-$\sigma$-stability is exactly $\sigma$-stability, which also means that $d$-decomposability and $\sigma$-stability coincide in this case.
\end{rem}

%#############################################################################################
%#############################################################################################

\section{On universal completeness and $B$-stability}

The aim of this section is to prove that the notion of universal completeness defined by a Boolean-valued metric coincides with the notion of $B$-stability defined by a regular equivalence relation. The main results of this section are Theorem \ref{thm.5.6} and Corollary \ref{coro.5.7}. For this, we first recall the related terminologies as follows.

Definition \ref{def.5.1} below is taken from \cite[Chapter 3, p.122]{KK}.

\begin{defn}\label{def.5.1}
	Let $X$ be a nonempty set and $B$ be a complete Boolean algebra. An ordered pair $(X,d)$ is called a Boolean set (brietly, a $B$-set) if $d$ is a mapping from $X\times X$ to $B$ such that the following three axioms are satisfied:
	\begin{itemize}
		\item [(1)] $d(x,y)=0$ iff $x=y$;
		\item [(2)]	$d(x,y)=d(y,x)$ for any $x$ and $y$ in $X$;
		\item [(3)]	$d(x,y)\leq d(x,z)\vee d(z,y)$ for any $x,y$ and $z$ in $X$.
	\end{itemize}
As usual, $d$ is called a $B$-valued metric on $X$.
\end{defn}

Let $(X,d)$ be a $B$-set and $G$ be a nonempty subset of $X$. $G$ is said to be cyclic if, for any subset $\{x_i,i\in I\}$ in $G$ and any partition $\{a_i,i\in I\}$ of unity in $B$ such that there exists $x\in X$ satisfying $a_i\wedge d(x,x_i)=0$ for each $i\in I$ (it is easy to check that such an $x$ is unique if it exists, called the mixing of $\{x_i,i\in I\}$ by $\{a_i,i\in I\}$ and denoted by $mix(a_i x_i)$), then $x$ must belong to $G$. $G$ is said to be universally complete or extended if $mix(a_i x_i)$ always exists and belongs to $G$ for each subset $\{x_i,i\in I\}$ in $G$ and each partition $\{a_i,i\in I\}$ of unity in $B$.

\begin{ex}\label{ex.5.2}
	Let $(S,D)$ be an RM space with base $(\Omega,\mathcal{F},\mu)$. For any $x$ and any $y$ in $S$, let $d(x,y)$ be $[D(x,y)>0]:=$ the equivalence class of $\{\omega \in \Omega: D^{0}(x,y)(\omega)>0\}$, where $D^{0}(x,y)$ is an arbitrarily chosen representative of $D(x,y)$, then it is easy to check that $(S,d)$ is a $B$-set and it is also obvious that a nonempty subset $G$ of $S$ is universally complete with respect to $d$ iff $G$ is $D$-$\sigma$-stable by noticing that for every partition $\{a_i,i\in I\}$ of unity in $B_{\mathcal{F}}$, $a_i=0$ for all but at almost countably many $i\in I$.
\end{ex}

Definition \ref{def.5.3} below gives an equivalent formulation of the notion of a conditional set, see Remark \ref{rem.5.5} below for details.

\begin{defn}\label{def.5.3}
	Let $X$ be a nonempty set and $B$ be a complete Boolean algebra. An equivalence relation $\sim$ on $X\times B$ (denote the equivalence class of $(x,a)$ by $x|a$ for each $(x,a)\in X\times B$) is said to be regular if the following conditions are satisfied:
	\begin{itemize}
		\item [(1)] $x|a=y|b$ implies $a=b$;
		\item [(2)]	$x|a=y|a$ implies $x|b=y|b$ for any $b\leq a$;
		\item [(3)]	$\{a\in B: x|a=y|a\}$ has a greatest element for any $x$ and $y$ in $X$;
		\item [(4)]	$x|1=y|1$ implies $x=y$.
	\end{itemize}
	In addition, given a regular equivalence relation on $X\times B$, a nonempty subset $G$ of $X$ is said to be $B$-stable with respect to $\sim$ if, for each nonempty subset $\{x_i,i\in I\}$ in $G$  and each partition $\{a_i,i\in I\}$ of unity in $B$, there exists $x\in G$ such that $x|a_i=x_i|a_i$ for each $i\in I$ (by (3) and (4) it is easy to see that such an $x$ is unique).
\end{defn}

\begin{ex}\label{ex.5.4}
	Let $S$ be an $L^0(\mathcal{F},\mathbb{K})$-module. Define an equivalence relation $\sim$ on $S\times B_{\mathcal{F}}$ by $(x,a)\sim (y,b)$ iff $a=b$ and $I_ax=I_ay$, then $\sim$ is regular and a nonempty subset $G$ of $S$ is $\sigma$-stable iff $G$ is $B_{\mathcal{F}}$-stable with respect to $\sim$. Thus the notion of $B$-stability can be regarded as a natural generalization of $\sigma$-stability.
\end{ex}
\begin{proof}
	We only need to check that $\sim$ is regular since the other assertions are obvious. In fact, we only need to verify (3) of Definition \ref{def.5.3} since (1),(2) and (4) of Definition \ref{def.5.3} are also clear in the case of this example. Let $M=\{a\in B_{\mathcal{F}}: x|a=y|a\}$ and $b=\bigvee M$, then by the particular property of the measure algebra $B_{\mathcal{F}}$ there exists a sequence $\{a'_{n},n\in \mathbb{N}\}$ in $M$ such that $b=\bigvee_{n}a'_{n}$. Further, let $a_1=a'_{1}$ and $a_n=a'_{n} \wedge ( \vee\{a'_{k}:k=1\sim n-1\} )^c$ for any $n\geq 2$, then $\{a_n,n\in \mathbb{N}\}\cup \{b^c\}$ forms a partition of unity in $B_{\mathcal{F}}$, it is also obvious that $I_{a_n}(I_bx)=I_{a_n} x=I_{a_n} y=I_{a_n} (I_by)$ for each $n\in \mathbb{N}$ and $I_{b^c}(I_bx)=0=I_{b^c}(I_by)$, and hence $I_bx=I_by$ by the regularity of $S$, namely $b$ is the greatest element of $M$.
\end{proof}

\begin{rem}\label{rem.5.5}
	In \cite{DJKK}, a conditional set of a nonempty set $X$ and a complete Boolean algebra $B$ is defined as a collection of objects $x|a$ for $x\in X$ and $a\in B$ such that the following conditions are satisfied:
	\begin{itemize}
		\item [(C1)] if $x|a=y|b$, then $a=b$;
		\item [(C2)] if $x,y\in X$ and $a,b\in B$ with $a\leq b$, then $x|b=y|b$ implies $x|a=y|a$;
		\item [(C3)] if $\{a_i,i\in I\}$ is a partition of unity in $B$ and  $\{x_i,i\in I\}$ is any nonempty subset of $X$, then there exists exactly one element $x\in X$ such that $x|a_i=x_i|a_i$ for each $i\in I$.		
	\end{itemize}
	It has been pointed out in \cite{DJKK} that the uniqueness of $x$ in (C3) implies $y=z$ whenever $y|1=z|1$, which is just (4) of Definition \ref{def.5.3}, in particular $ii)$ of \cite[Proposition 2.4]{DJKK} shows that (3) of Definition \ref{def.5.3} holds by (C2) and (C3). Thus the notion of a conditional set of $X$ and $B$ amounts to $X$ itself being $B$-stable with respect to the equivalence relation $\sim$ defined by $(x,a)\sim (y,b)$ if $x|a=y|b$, the former emphasizes the set $\{x|a:x\in X ~$and$~ a\in B\}$ of equivalence classes, whereas the latter emphasizes the $B$-stable set $X$ itself. Here, we first give the notion of a regular equivalence relation and then define a $B$-stable set with respect to such an equivalence relation in order to provide a convenient connection with the notion of a universally complete set. In particular, we have the following result--Theorem \ref{thm.5.6}.
\end{rem}

\begin{thm}\label{thm.5.6}
	Let $(X,d)$ be a $B$-set. Define an equivalence relation $\sim$ on $X\times B$ by $(x,a)\sim (y,b)$ iff $a=b$ and $a\wedge d(x,y)=0$, then the equivalence relation is regular, called the equivalence relation induced by $d$. Conversely, any regular equivalence relation $\sim$ on $X\times B$ can be induced by a $B$-valued metric on $X$.
\end{thm}
\begin{proof}
	If we denote by $x|a$ the equivalence class of $(x,a)$ under the equivalence relation $\sim$ induced by $d$, then it is very easy to verify that $\sim$ is indeed an equivalence relation and (1), (2) and (4) in Definition \ref{def.5.3} are also obvious. As for (3) of Definition \ref{def.5.3}, let $M=\{a\in B:x|a=y|a\}$ and $b=\bigvee M$, then $a\wedge d(x,y)=0$ for each $a\in M$, and further by the distributive law in $B$ $b\wedge d(x,y)=\bigvee \{a\wedge d(x,y):a\in M\}=0$, namely $x|b=y|b$, that is to say, $b$ is the greatest element of $\{a\in B:x|a=y|a\}$, so $\sim$ is regular.
	
	Conversely, let $\sim$ be a regular equivalence relation on $X\times B$ and denote by $x|a$ the equivalence class of $(x,a)$ under $\sim$, defined $d:X\times X\rightarrow B$ as follows:
	$$d(x,y)=( \max\{a\in B:x|a=y|a\} )^c~\text{for any}~x ~\text{and}~y~\text{in}~X,$$
	where $b^c$ stands for the complement of $b$ for any given $b\in B$.
	One can easily check that $d$ is a $B$-valued metric on $X$ and the equivalence relation induced by $d$ is exactly $\sim$.
\end{proof}

\begin{cor}\label{coro.5.7}
	The notion of universal completeness defined by a Boolean-valued metric coincides with the notion of $B$-stability defined by a regular equivalence relation.
\end{cor}

\begin{rem}\label{rem.5.8}
It is easy to see that the notion of a cyclic set includes as a special case the notion of a set having the relative countable concatenation property (also called a relatively $\sigma$-stable set, introduced independently in \cite{WG,Zap}). It is also obvious that the notion of a cyclic set is weaker than the notion of a universally complete set, and it is the latter that is frequently used in the study of related issues, so here we do not intend to make too much discussion about the notion of a cyclic set.
\end{rem}
%####################################################################################################
%####################################################################################################

% ------------------------------------------------------------------------
\end{document}